\documentclass{amsart}
\usepackage{amssymb}
\usepackage{braket}
\usepackage{mathrsfs}
\usepackage[all]{xy}
\usepackage{graphicx}

\newcommand{\tab}{\frac{\beta}{\tau}} %%%\newcommand{\tab}{\tau\act \beta}
\newcommand{\tabp}{\frac{\beta\ppr}{\tau\ppr}} %%%\newcommand{\tabp}{\tau\ppr\ac

\newtheorem{thm}{Theorem}[section]
\newtheorem{cor}[thm]{Corollary}
\newtheorem{prop}[thm]{Proposition}
\theoremstyle{definition}
\newtheorem{dfn}[thm]{Definition}
\newtheorem{ex}[thm]{Example}

\newtheorem{lem}[thm]{Lemma}
\newtheorem{fact}[thm]{Fact}

\theoremstyle{remark}
\newtheorem{rem}[thm]{Remark}
\newcommand{\Ob}{\mathrm{Ob}}         %%% Ob
\newcommand{\id}{\mathrm{id}}         %%% id
\newcommand{\Id}{\mathrm{Id}}         %%% Id
\newcommand{\ppr}{^{\prime}}          %%% '
\newcommand{\pprr}{^{\prime\prime}}   %%% ''
\newcommand{\Hom}{\mathrm{Hom}}       %%% Hom
\newcommand{\pro}{\mathrm{pr}}        %%% pr（※pro）
\newcommand{\op}{\mathrm{op}}         %%% op
\newcommand{\cl}{\mathfrak{C}}        %%% cl
\newcommand{\fa}{\forall}             %%% 任意記号
\newcommand{\am}{\amalg}              %%% 集合の直和
\newcommand{\co}{\colon}              %%% コロン
\newcommand{\ci}{\circ}               %%% 合成
\newcommand{\iv}{^{-1}}               %%% ^{-1}
\newcommand{\ot}{\otimes}             %%% テンソル
\newcommand{\otd}{\underset{d}{\ot}}             %%% テンソル
\newcommand{\se}{\subseteq}           %%% 包含
\newcommand{\ti}{\times}              %%% ×（直積）
\newcommand{\uas}{^{\ast}}            %%% ^*
\newcommand{\sas}{_{\ast}}            %%% _*
\newcommand{\ems}{\emptyset}          %%% 空集合
\newcommand{\Mon}{\mathit{Mon}}       %%% Mon
\newcommand{\Sett}{\mathit{Set}}      %%% Set (※Sett)
\newcommand{\sett}{\mathit{set}}      %%% set (※sett)
\newcommand{\Ab}{\mathit{Ab}}         %%% Ab
\newcommand{\Mod}{\mathit{Mod}}       %%% Mod
\newcommand{\Fun}{\mathit{Fun}}       %%% Fun
\newcommand{\kMod}{k\Mod}             %%% kMod
\newcommand{\CAT}{\mathrm{CAT}}       %%% CAT
\newcommand{\Cat}{\mathrm{Cat}}       %%% Cat

\newcommand{\lla}{\longleftarrow}     %%% 長い左矢印
\newcommand{\lra}{\longrightarrow}    %%% 長い右矢印
\newcommand{\rta}{\rightharpoonup}    %%% 上半右矢印
\newcommand{\lta}{\leftharpoonup}     %%% 上半左矢印
\newcommand{\tc}{\Rightarrow}         %%% 二重右矢印
\newcommand{\LR}{\Leftrightarrow}     %%% 同値矢印
\newcommand{\Aut}{\mathrm{Aut}}       %%% Aut

\newcommand{\al}{\alpha}         %%% α
\newcommand{\be}{\beta}          %%% β
\newcommand{\lam}{\lambda}       %%% λ
\newcommand{\ups}{\upsilon}      %%% υ
\newcommand{\ep}{\varepsilon}    %%% varε
\newcommand{\thh}{\theta}        %%% θ
\newcommand{\om}{\omega}         %%% ω
\newcommand{\vp}{\varphi}        %%% varΦ
\newcommand{\Lam}{\Lambda}       %%% Λ（大文字ラムダ）
\newcommand{\Gam}{\Gamma}        %%% Γ（大文字ガンマ）
\newcommand{\Om}{\Omega}         %%% Γ（大文字オメガ）

\newcommand{\Asc}{\mathscr{A}}  %%% scr
\newcommand{\Csc}{\mathscr{C}}  %%% scr
\newcommand{\Isc}{\mathscr{I}}  %%% scr
\newcommand{\Ksc}{\mathscr{K}}  %%% scr
\newcommand{\C}{\Csc}           %%% scr
\newcommand{\A}{\Asc}           %%% scr
\newcommand{\M}{\mathscr{M}^k}    %%% scr
\newcommand{\HH}{\mathscr{H}}   %%% scr

\newcommand{\Cbb}{\mathbb{C}}   %%% bb
\newcommand{\Kbb}{\mathbb{K}}   %%% bb
\newcommand{\Sbb}{\mathbb{S}}   %%% bb
\newcommand{\Zbb}{\mathbb{Z}}   %%% bb
\newcommand{\Dbf}{\mathbf{D}}   %%% bf
\newcommand{\Dbb}{\mathbb{D}}   %%% bf
\newcommand{\Rbf}{\mathbf{R}}   %%% bf
\newcommand{\Tbf}{\mathbf{T}}   %%% bf
\newcommand{\ebf}{\mathbf{e}}   %%% bf
\newcommand{\ibf}{\mathbf{i}}   %%% bf
\newcommand{\pbf}{\mathbf{p}}   %%% bf
\newcommand{\Bcal}{\mathcal{B}} %%% cal
\newcommand{\Fcal}{\mathcal{F}} %%% cal
\newcommand{\Gcal}{\mathcal{G}} %%% cal
\newcommand{\Hcal}{\mathcal{H}} %%% cal
\newcommand{\Ical}{\mathcal{I}} %%% cal
\newcommand{\Jcal}{\mathcal{J}} %%% cal
\newcommand{\Kcal}{\mathcal{K}} %%% cal
\newcommand{\Lcal}{\mathcal{L}} %%% cal
\newcommand{\Scal}{\mathcal{S}} %%% cal
\newcommand{\Ycal}{\mathcal{Y}} %%% cal
\newcommand{\afr}{\mathfrak{a}} %%% ドイツ
\newcommand{\bfr}{\mathfrak{b}} %%% ドイツ
\newcommand{\sfr}{\mathfrak{s}} %%% ドイツ
\newcommand{\Afr}{\mathfrak{A}} %%% ドイツ
\newcommand{\wt}{\widetilde}    %%% 広チルダ
\newcommand{\und}{\underline}   %%% 下線
\newcommand{\ov}{\overset}      %%% overset
\newcommand{\un}{\underset}     %%% underset

\newcommand{\Ind}{\mathrm{Ind}}    %%% Ind
\newcommand{\Gs}{{}_G\mathit{set}}           %%% Gset
\newcommand{\Hs}{{}_H\mathit{set}}           %%% Hset
\newcommand{\el}{\mathbb{E}}

\newcommand{\iog}{\iota^{(G)}}
\newcommand{\ioh}{\iota^{(H)}}
\newcommand{\prg}{\pro^{(G)}}
\newcommand{\prh}{\pro^{(H)}}
\newcommand{\ax}{\al(x)}        %%%
\newcommand{\wlS}{\frac{W_S}{L_S}}
\newcommand{\wlT}{\frac{W_T}{L_T}}

\newcommand{\spSu}{\yh\ov{\und{\be}_S}{\lla}\wlS\ov{\und{\al}_S}{\lra}\xg}
\newcommand{\spTu}{\yh\ov{\und{\be}_T}{\lla}\wlT\ov{\und{\al}_T}{\lra}\xg}

\newcommand{\spStu}{\und{\be}_S,\wlS,\und{\al}_S}

\newcommand{\pt}{\mathbf{1}}                     %%% １（太字）
\newcommand{\ttt}{\mathbf{t}}                    %%% t（太字）
\newcommand{\iii}{\mathfrak{i}}                  %%% i（fraktur）
\newcommand{\ppp}{\mathfrak{p}}                  %%% p（fraktur）
\newcommand{\dfl}{\mathrm{dfl}}                  %%% dfl
\newcommand{\nwp}{\mathrm{nwp}}                  %%% nwp
\newcommand{\SIm}{\mathrm{SIm}}                  %%% SIm
\newcommand{\Obig}{\Omega_{\mathrm{big}}}        %%% Ωbig
\newcommand{\Abig}{\Afr_{\mathrm{big}}}          %%% Ωbig
\newcommand{\as}{\mathrm{as}}                    %%% as
\newcommand{\sym}{\mathrm{sym}}                  %%% sym
\newcommand{\ac}{\mathrm{ac}}                    %%% ac
\newcommand{\Omk}{\Om^k}                         %%% Ω^k
\newcommand{\Obigk}{\Omk_{\mathrm{big}}}         %%% Ωbig^k
\newcommand{\SMack}{\mathit{SMack}}           %%% SMack
\newcommand{\Mack}{\mathit{Mack}}             %%% Mack
\newcommand{\SMackS}{\SMack(\Sbb)}            %%% SMack(S)
\newcommand{\SMackC}{\SMack(\C)}              %%% SMack(C)
\newcommand{\MackSk}{\Mack^k(\Sbb)}           %%% Mack^k(C)
\newcommand{\MackdSk}{\Mack_{\dfl}^k(\Sbb)}   %%% Mack_dfl^k(C)
\newcommand{\MackCk}{\Mack^k(\C)}             %%% Mack^k(C)
\newcommand{\MackdCk}{\Mack_{\dfl}^k(\C)}     %%% Mack_dfl^k(C)

\newcommand{\finGpd}{\mathrm{finGpd}}         %%% finGpd
\newcommand{\finCat}{\mathrm{finCat}}         %%% finCat
\newcommand{\BFk}{\Fcal_{\Bcal}^k}            %%% k-線型Biset関手の圏

\newcommand{\Green}{\mathit{Green}}           %%% Green
\newcommand{\GreenCk}{\Green(\C)}             %%% Green_dfl(C)
\newcommand{\GreendCk}{\Green_{\dfl}^k(\C)}   %%% Green_dfl^k(C)

\newcommand{\Add}{\mathit{Add}}               %%% Add
\newcommand{\Sp}{\mathrm{Sp}}                 %%% スパン圏
\newcommand{\OMod}{\Omk\Mod}                  %%% ΩMod
\newcommand{\Cmod}{\Cbb\mathrm{mod}}          %%% Cmod
\newcommand{\xg}{\frac{X}{G}}

\newcommand{\yh}{\frac{Y}{H}}
\newcommand{\yg}{\frac{Y}{G}}

\newcommand{\zk}{\frac{Z}{K}}

\newcommand{\ak}{\frac{A}{K}}
\newcommand{\akp}{\frac{A\ppr}{K\ppr}}

\newcommand{\bl}{\frac{B}{L}}
\newcommand{\blp}{\frac{B\ppr}{L\ppr}}

\newcommand{\wl}{\frac{W}{L}}

\newcommand{\pte}{\frac{\pt}{e}}
\newcommand{\ptg}{\frac{\pt}{G}}
\newcommand{\pth}{\frac{\pt}{H}}
\newcommand{\ptk}{\frac{\pt}{K}}

\newcommand{\althh}{\frac{\alpha}{\theta}}

\newcommand{\althhp}{\frac{\alpha\ppr}{\theta\ppr}}
\newcommand{\bet}{\frac{\beta}{\tau}}

\newcommand{\xgixg}{\xg\ov{\id}{\to}\xg}

\newcommand{\akaxgu}{\ak\ov{\und{\afr}}{\to}\xg}
\newcommand{\akaxgpu}{\akp\ov{\und{\afr\ppr}}{\to}\xg}

\newcommand{\blbyhu}{\bl\ov{\und{\bfr}}{\to}\yh}

\numberwithin{equation}{section}

\begin{document}
%\subjclass[2000]{,}

\title[Biset functors as module Mackey functors]{Biset functors as module Mackey functors, and its relation to derivators.}

\author{Hiroyuki NAKAOKA}
\address{Research and Education Assembly, Science and Engineering Area, Research Field in Science, Kagoshima University, 1-21-35 Korimoto, Kagoshima, 890-0065 Japan\ /\ LAMFA, Universit\'{e} de Picardie-Jules Verne, 33 rue St Leu, 80039 Amiens Cedex1, France}

\email{nakaoka@sci.kagoshima-u.ac.jp}

%\thanks{The author wishes to thank Professor Fumihito Oda for stimulating arguments and useful comments}
\thanks{The author wishes to thank Professor Serge Bouc for his comments and advices}
\thanks{The author wishes to thank Professor Ivo Dell'Ambrogio, for his interest and stimulating discussions}

\thanks{This work is supported by JSPS KAKENHI Grant Numbers 25800022,\, 24540085.}

\begin{abstract}
In this article, we will show that the category of biset functors can be regarded as a reflective monoidal subcategory of the category of Mackey functors on the 2-category of finite groupoids. This reflective subcategory is equivalent to the category of modules over the Burnside functor.
As a consequence of the reflectivity, we can associate a biset functor to any derivator on the 2-category of finite categories.
\end{abstract}

\maketitle

%\tableofcontents

\section{Introduction and Preliminaries}

A Mackey functor is a useful tool to describe how an algebraic system associated to finite groups (such as the Burnside rings or the representation rings, or the cohomology groups, etc.) behaves under the change of subgroups of a fixed group $G$.
Recently, Bouc \cite{Bouc_biset} has defined the notion of a {\it biset functor}, which moreover enables us to deal with the behavior of algebraic systems named as above, with respect to {\it bisets} among {\it all} finite groups.
The category of biset functors $\BFk$ becomes a symmetric closed monoidal category (\cite[II.8]{Bouc_biset}), and its monoid objects called {\it Green biset functors}, are also investigated in the literature \cite{Bouc_biset},\cite{Romero}.

As shown in \cite{N_BisetMackey}, a biset functor can be regarded as a special class of Mackey functor on some 2-category $\Sbb$. This special Mackey functors are called {\it deflative} Mackey functors there. Indeed, the category of biset functors $\BFk$ has shown to be equivalent to the category $\MackdSk$ of deflative Mackey functors on $\Sbb$. Relations between (global) Mackey functors and biset functors are also investigated by several researchers such as \cite{Bouc_fused}, \cite{Coskun}, \cite{HTW}, \cite{Ibarra}.

We can also show that the 2-category $\Sbb$ becomes biequivalent to the 2-category of finite groupoids. It allows us to relate $\Mack(\Sbb)$ to {\it derivators}. In fact, to any derivator on the 2-category of finite categories, we can associate a Mackey functor on $\Sbb$. This gives a way to construct biset functors from derivators (since the inclusion $\BFk\hookrightarrow\MackSk$ is shown to have a left adjoint, as below). We remark also that the theory of derivators is of recent interest by several researchers, as seen in \cite{Cisinski}, \cite{Groth}.

One of the motivations of this article is to pursue this interpretation, to describe Green biset functors in terms of $\MackSk$. By the result of Panchadcharam and Street \cite{PS}, the category $\MackSk$ can be equipped with a symmetric monoidal structure. We will show that this monoidal structure restricts to $\MackdSk$, and the equivalence $\BFk\ov{\simeq}{\lra}\MackdSk$ constructed in \cite{N_BisetMackey} is in fact a monoidal equivalence. As a consequence, their categories of monoids become equivalent. Namely, we obtain an equivalence between the category of Green biset functors and the category of {\it deflative Green functors} on $\Sbb$. Besides, we see that the inclusion $\BFk\hookrightarrow\MackSk$ has a left adjoint, given by the tensor product with a monoid object.

This motivation comes from our ongoing work on {\it Tambara biset functors}. It is expected to be a framework for biset functors equipped with multiplicative inductions, such as the Burnside functor and the representation ring functor. This will be an analog of the notion of {\it Tambara functors} in the ordinary Mackey functor theory. In \cite{N_DerTam} and forthcoming works, we will formalize \lq Tambara' properties for biset functors, using our interpretation as Mackey functors on $\Sbb$. In the analogy with the ordinary case, we expect that the underlying Mackey functor of any Tambara biset functor should naturally become a Green functor on $\Sbb$. In this article, we develop what will be necessary for this purpose.

\bigskip

\bigskip

Throughout this article, $k$ denotes a commutative ring with multiplicative unit. Any group is assumed to be finite. The unit of a group is denoted by $e$. Abbreviately, trivial group is denoted by $e$. For a finite group $G$, let $\Gs$ denote the category of finite $G$-sets, where morphisms are equivariant $G$-maps. The discrete category with only one object is denoted by $\ebf$.
For any category $\Ksc$ and any pair of objects $X$ and $Y$ in $\Ksc$, the set of morphisms from $X$ to $Y$ in $\Ksc$ is denoted by $\Ksc(X,Y)$. Let $\cl(\Ksc)$ denote the set of isomorphism classes of objects in $\Ksc$. For a functor $F\co\Ksc\to\Ksc\ppr$, let $\cl(F)\co\cl(\Ksc)\to\cl(\Ksc\ppr)$ denote the map induced from $F$.

\bigskip

For a strict 2-category $\Kbb$, we use the following notation.
\begin{enumerate}
\item[{\rm (0)}] $\Kbb^0$ denotes the class of 0-cells in $\Kbb$.
\item[{\rm (1)}] For any pair of 0-cells $X,Y\in\Kbb^0$, the set of 1-cells from $X$ to $Y$ is denoted by $\Kbb^1(X,Y)$. Together with the 2-cells among them, they form a category $\Kbb(X,Y)$ satisfying $\Ob(\Kbb(X,Y))=\Kbb^1(X,Y)$.
\item[{\rm (2)}] For any $X,Y\in\Kbb^0$ and any $f,g\in\Kbb^1(X,Y)$, the set of 2-cells from $f$ to $g$ is denoted by $\Kbb^2(f,g)$. Namely, we put $\Kbb^2(f,g)=(\Kbb(X,Y))(f,g)$.
\end{enumerate}
Horizontal composition and vertical composition are denoted by \lq\lq$\ci$" and \lq\lq$\cdot$" respectively.

Let $\Cat$ denote the 2-category of small categories (\cite[Example 7.1.4a]{Borceux}). Let $\finCat\se\Cat$ denote the full 2-subcategory of finite categories (this is an example of {\it category of diagrams} (\cite[Definition 1.12]{Groth})), and let $\finGpd\se\finCat$ denote the full 2-subcategory of finite groupoids.

\bigskip

Let us recall the definition and properties of $\Sbb$ from \cite{N_BisetMackey}. For the generalities of 2-categries and bicategories, see \cite{MacLane},\cite{Borceux},\cite{Leinster}.
\begin{dfn}\label{DefS}(\cite[Definition 2.2.12]{N_BisetMackey})
2-category $\Sbb$ is defined as follows. (See \cite[section 2.2]{N_BisetMackey}.)
\begin{enumerate}
\item[{\rm (0)}] A 0-cell is a pair of a finite group $G$ and a finite $G$-set $X$. We denote this pair by $\xg$.
\item[{\rm (1)}] For any pair of 0-cells $\xg$ and $\yh$, a 1-cell $\althh\co \xg\to\yh$ is a pair of a map $\al\co X\to Y$ and a family of maps $\{\thh_x\co G\to H \}_{x\in X}$ satisfying
\begin{itemize}
\item[{\rm (i)}] $\al(gx)=\thh_x(g)\al(x)$ 
\item[{\rm (ii)}] $\thh_x(gg\ppr)=\thh_{g\ppr x}(g)\thh_x(g\ppr)$
\end{itemize}
for any $x\in X$ and any $g,g\ppr\in G$.
We denote such a pair by $\althh$.

If $\thh$ satisfies $\thh_x=f\ (\fa x\in X)$ for some group homomorphism $f\co G\to H$, then the 1-cell $\althh$ is called $f$-{\it equivariant}, and simply written as $\frac{\al}{f}$.
In particular when $G=H$, a 1-cell $\althh\co\xg\to\yg$ is called $G$-{\it equivariant} or simply {\it equivariant}, if it is $\id_G$-equivariant. In this case, we denote this 1-cell by $\frac{\al}{\id_G}=\frac{\al}{G}$.

For any consecutive pair of 1-cells
\[ \xg\ov{\althh}{\lra}\yh\ov{\tab}{\lra}\zk, \]
we define their composition $(\tab)\ci(\althh)=\frac{\be\ci\al}{\tau\ci\thh}$ by
\begin{itemize}
\item[-] $\be\ci\al\co X\to Z$ is the usual composition of maps of sets,
\item[-] $\tau\ci\thh$ is defined by
\[ (\tau\ci\thh)_x(g)=\tau_{\ax}(\thh_x(g))\quad(\fa g\in G), \]
namely, $(\tau\ci\thh)_x=\tau_{\ax}\ci\thh_x$ for any $x\in X$.
\end{itemize}
The identity 1-cell for $\xg$ is given by $\id_{\xg}=\frac{\id_X}{G}$.

\item[{\rm (2)}] For any pair of 1-cells $\althh,\althhp\co\xg\to\yh$, a 2-cell $\ep\co\althh\tc\althhp$ is a family of elements $\{ \ep_x\in H\}_{x\in X}$ satisfying
\begin{itemize}
\item[{\rm (i)}] $\al\ppr(x)=\ep_x\al(x) $,
\item[{\rm (ii)}] $\ep_{gx}\thh_x(g)\ep_x\iv=\thh\ppr_x(g)$
\end{itemize}
for any $x\in X$ and $g\in G$.

If we are given a consecutive pair of 2-cells
\[
\xy
(-14,0)*+{\xg}="0";
(14,0)*+{\yh}="2";
{\ar@/^2.0pc/^{\althh} "0";"2"};
{\ar|*+{_{\althhp}} "0";"2"};
{\ar@/_2.0pc/_{\frac{\al\pprr}{\thh\pprr}} "0";"2"};
{\ar@{=>}^{\ep} (0,6);(0,3)};
{\ar@{=>}^{\ep\ppr} (0,-3);(0,-6)};
\endxy
\]
then their vertical composition $\ep\ppr\cdot\ep\co \althh\tc \frac{\al\pprr}{\thh\pprr}$ is defined by
\[ (\ep\ppr\cdot\ep)_x=\ep\ppr_x\ep_x\quad(\fa x\in X). \]
The identity 2-cell $\id\co\id_{\althh}\tc\id_{\althh}$ is given by
$\id_x=e\ (\fa x\in X)$. With this definition, any 2-cell becomes invertible with respect to the vertical composition.

Horizontal compositions are given as follows.
Let $\xg\ov{\althh}{\lra}\yh\ov{\tab}{\lra}\zk$ be a sequence of 1-cells.
\begin{enumerate}
\item For a 2-cell
\[
\xy
(-14,0)*+{\xg}="0";
(14,0)*+{\yh}="2";
{\ar@/^1.2pc/^{\althh} "0";"2"};
{\ar@/_1.2pc/_{\althhp} "0";"2"};
{\ar@{=>}^{\ep} (0,2);(0,-2)};
\endxy,
\]
define $(\tab)\ci\ep\co(\tab)\ci(\althh)\tc(\tab)\ci(\althhp)$ by
\begin{equation}\label{EqHor1}
((\tab)\ci\ep)_x=\tau_{\ax}(\ep_x)\quad(\fa x\in X).
\end{equation}
\item For a 2-cell
\[
\xy
(-14,0)*+{\yh}="0";
(14,0)*+{\zk}="2";
{\ar@/^1.2pc/^{\tab} "0";"2"};
{\ar@/_1.2pc/_{\tabp} "0";"2"};
{\ar@{=>}^{\rho} (0,2);(0,-2)};
\endxy,
\]
define $\rho\ci(\althh)\co(\tab)\ci(\althh)\tc(\tabp)\ci(\althh)$ by
\begin{equation}\label{EqHor2}
(\rho\ci(\althh))_x=\rho_{\ax}\quad(\fa x\in X).
\end{equation}
\end{enumerate}
\end{enumerate}
\end{dfn}
\begin{rem}\label{RemAbb}
In the following, a 1-cell $\althh$ is often abbreviately written as $\al$.
\end{rem}

Define a 2-functor $\el\co\Sbb\to\finGpd$ in the following way (Definition \ref{Def0to0}, Propositions \ref{Prop1to1}, \ref{Prop2to2}, Corollary \ref{Cor2Ftr}).

On 0-cells, $\el$ is defined as follows.
\begin{dfn}\label{Def0to0}
For any 0-cell $\xg\in\Sbb^0$, associate the category of elements $\el(\xg)=e\ell_G(X)$ (\cite[Example A.14]{AM}) defined by
\begin{itemize}
\item[-] $\Ob(\el(\xg))=X$.
\item[-] For any $x,x\ppr\in X$, define $(\el(\xg))(x,x\ppr)=\{ g\in G\mid gx=x\ppr\}$.
\end{itemize}
\end{dfn}

On 1-cells, $\el$ is given by the following.
\begin{prop}\label{Prop1to1}
Let $\xg,\yh\in\Sbb^0$ be any pair of 0-cells. There is a bijective map on the set of 1-cells
\[ \el\co\Sbb^1(\xg,\yh)\ov{\cong}{\lra}\finGpd^1(\el(\xg),\el(\yh)), \]
which is given explicitly in the proof.
\end{prop}
\begin{proof}
Any functor $F\co \el(\xg)\to \el(\yh)$ should consist of a pair $(\alpha,\{ F_{x,x\ppr}\}_{x,x\ppr\in X})$ of
\begin{itemize}
\item[{\rm (i)}] a map on the set of objects $\al\co X\to Y$,
\item[{\rm (ii)}] a family of maps
\begin{equation}\label{FamMorphMap}
\{ F_{x,x\ppr}\co(\el(\xg))(x,x\ppr)\to (\el(\yh))(\al(x),\al(x\ppr)) \}
\end{equation}
preserving the identities and the composition.
\end{itemize}

A family of maps $(\ref{FamMorphMap})$ can be regarded as a map
\begin{eqnarray*}
&\Theta\co X\ti G\ov{\cong}{\to}\{(x,x\ppr,g)\in X\ti X\ti G\mid x\ppr=gx \}\to H&\\
&(x,g)\mapsto(x,gx,g)\mapsto F_{x,gx}(g)=:\thh_x(g).&
\end{eqnarray*}
This family preserves the composition if and only if, for any sequence of morphisms $x\ov{g\ppr}{\to}g\ppr x\ov{g}{\to}gg\ppr x$ in $\el(\xg)$, it satisfies
\[  \thh_x(gg\ppr)=\thh_{g\ppr x}(g)\ci \thh_x(g\ppr).\]
This is nothing but the condition for a 1-cell $\althh\co\xg\to\yh$ in $\Sbb$. Preservation of the identities also follows automatically from this equation.
Thus functors $F$ correspond bijectively to 1-cells $\althh$ by the equation
\[ F_{x,gx}(g)=\thh_x(g)\quad(\fa x\in X,g\in G). \]
\end{proof}

On 2-cells, $\el$ is given by the following.
\begin{prop}\label{Prop2to2}
Let $\althh,\bet\co\xg\to\yh$ be any pair of 1-cells in $\Sbb$. There is a bijective map on the set of 2-cells
\[ \el\co\Sbb^2(\althh,\bet)\ov{\cong}{\lra}\finGpd^2(\el(\althh),\el(\bet)), \]
which is given explicitly in the proof.
\end{prop}
\begin{proof}
Remark that each element in both hand sides is given as a family $\ep=\{ \ep_x\in H\}_{x\in X}$ satisfying some conditions.

For the left hand side, the conditions are
\begin{itemize}
\item[{\rm (i)}] $\ep_x\al(x)=\be(x)\ \ (\fa x\in X)$.
\item[{\rm (ii)}] $\ep_{gx}\thh_x(g)=\tau_x(g)\ep_x\ \ (\fa x\in X,g\in G)$.
\end{itemize}

For the right hand side, the conditions are
\begin{itemize}
\item[{\rm (i)}] $\ep_x\in(\el(\yh))(\al(x),\be(x))\ \ (\fa x\in X)$.
\item[{\rm (ii)}] For any $x\in X$ and any $g\in G$ (viewed as a morphism $g\co x\to gx$),
\[
\xy
(-8,6)*+{\al(x)}="0";
(8,6)*+{\be(x)}="2";
(-8,-6)*+{\al(gx)}="4";
(8,-6)*+{\be(gx)}="6";
{\ar^{\ep_x} "0";"2"};
{\ar_{\thh_x(g)} "0";"4"};
{\ar^{\tau_x(g)} "2";"6"};
{\ar_{\ep_{gx}} "4";"6"};
{\ar@{}|\circlearrowright "0";"6"};
\endxy
\]
is commutative.
\end{itemize}
Obviously, these conditions are equivalent.
\end{proof}

\begin{cor}\label{Cor2Ftr}
$\el\co\Sbb\to\finGpd$ is a strict 2-functor $($\cite[Definition 7.2.1]{Borceux}$)$. This induces isomorphism of categories $\el\co\Sbb(\xg,\yh)\to\finGpd(\el(\xg),\el(\yh))$ for any $\xg,\yh\in\Sbb^0$.
\end{cor}
\begin{proof}
It is straightforward to check $\el$ is in fact a 2-functor. The latter part also follows from Propositions \ref{Prop1to1}, \ref{Prop2to2} immediately.
\end{proof}

\begin{dfn}\label{DefInd}$($\cite[Definition 3.1.1]{N_BisetMackey}$)$
Let $\iota\co H\hookrightarrow G$ be a monomorphism of groups.
For any $X\in\Ob(\Hs)$, we define $\Ind_{\iota}X\in\Ob(\Gs)$ by
\[ \Ind_{\iota}X=(G\times X)/\sim, \]
where the equivalence relation $\sim$ is defined by
\begin{itemize}
\item[-] $(\xi,x)$ and $(\xi\ppr,x\ppr)$ in $G\times X$ are equivalent if there exists $h\in H$ satisfying
\[ x\ppr=hx,\ \ \xi=\xi\ppr\iota(h). \]
\end{itemize}
We denote the equivalence class of $(\xi,x)$ by $[\xi,x]\in\Ind_{\iota}X$. The $G$-action on $\Ind_{\iota}X$ is defined by
\[ g[\xi,x]=[g\xi,x] \]
for any $g\in G$ and $[\xi,x]\in\Ind_{\iota}X$.
\end{dfn}

The following facts have been shown in \cite{N_BisetMackey}.
\begin{fact}\label{PropIndEquiv}$($\cite[Proposition 3.1.2]{N_BisetMackey}$)$
Let $\iota\co H\hookrightarrow G$ be a monomorphism of groups.
For any $X\in\Ob(\Hs)$, if we define a map $\ups\co X\to\Ind_{\iota}X$ by
\[ \ups(x)=[e,x]\quad(\fa x\in X), \]
then the 1-cell
\[ \frac{\ups}{\iota}\co\frac{X}{H}\to \frac{\Ind_{\iota}X}{G} \]
becomes an equivalence $($cf. Remark \ref{RemImmed}$)$.
\end{fact}

\begin{fact}\label{Prop2CoprodEqui}$($\cite[Proposition 3.2.13]{N_BisetMackey}$)$
Let $G$ be any finite group.
For any $X,Y\in\Ob(\Gs)$, let
\[ X\ov{\ups_X}{\hookrightarrow}X\am Y\ov{\ups_Y}{\hookleftarrow}Y \]
be the coproduct in $\Gs$.
Then
\begin{equation}\label{Equ1}
\xg\ov{\frac{\ups_X}{G}}{\lra}\frac{X\am Y}{G}\ov{\frac{\ups_Y}{G}}{\lla}\frac{Y}{G}
\end{equation}
gives a bicoproduct of $\xg$ and $\yg$ in $\Sbb$.
\end{fact}

\begin{rem}\label{Rema1}
$\el$ sends $(\ref{Equ1})$ to the coproduct of categories
\[ \el(\xg)\hookrightarrow\el(\xg)\am\el(\yh)\hookleftarrow\el(\yg). \]
\end{rem}

\begin{fact}\label{Prop2CoprodVari}$($\cite[Proposition 3.2.15]{N_BisetMackey}$)$
Let $\xg$ and $\yh$ be any pair of 0-cells in $\Sbb$. Denote the monomorphisms\begin{eqnarray*}
&G\to G\times H\ ; \ g\mapsto (g,e)&\\
&H\to G\times H\ ; \ h\mapsto (e,h)&
\end{eqnarray*}
by $\iog$ and $\ioh$ respectively, and denote the natural maps
\begin{eqnarray*}
&X\to\Ind_{\iog}X\am\Ind_{\ioh}Y\ ;\ x\mapsto [e,x]\in\Ind_{\iog}X&\\
&Y\to\Ind_{\iog}X\am\Ind_{\ioh}Y\ ;\ y\mapsto [e,y]\in\Ind_{\ioh}Y&
\end{eqnarray*}
by $\ups_X$ and $\ups_Y$.
Then
\[ \xg\ov{\frac{\ups_X}{\iog}}{\lra}\frac{\Ind_{\iog}X\am\Ind_{\ioh}Y}{G\times H}\ov{\frac{\ups_Y}{\ioh}}{\lla}\yh \]
gives a bicoproduct $\xg\am\yh$ of $\xg$ and $\yh$ in $\Sbb$.
\end{fact}

\begin{rem}\label{RemCCCCC}
As a consequence, for any $\xg,\yh\in\Sbb^0$, its image $\el(\xg\am\yh)$ becomes equivalent to the coproduct of $\el(\xg)$ and $\el(\yh)$ by
\begin{eqnarray*}
\el(\xg\am\yh)&=&\el(\frac{\Ind_{\iog}X\am\Ind_{\ioh}Y}{G\ti H})\\
&=&\el(\frac{\Ind_{\iog}X}{G\ti H})\, \am\, \el(\frac{\Ind_{\ioh}Y}{G\ti H})%
\ \simeq\ \el(\xg)\am\el(\yh).
\end{eqnarray*}
\end{rem}

\begin{fact}\label{Prop2Pullback}$($\cite[Proposition 3.2.17]{N_BisetMackey}$)$
Let $\althh\co\xg\to\zk$ and $\bet\co\yh\to\zk$ be any pair of 1-cells in $\Sbb$.
Denote the natural projection homomorphisms by
\[ \prg\co G\times H\to G,\ \ \prh\co G\times H\to H. \]
If we
\begin{itemize}
\item[-] put $F=\{(x,y,k)\in X\times Y\times K\mid \be(y)=k\ax \}$, and put
\begin{eqnarray*}
&\wp_X\co F\to X\ ;\ (x,y,k)\mapsto x,&\\
&\wp_Y\co F\to Y\ ;\ (x,y,k)\mapsto y,&
\end{eqnarray*}
\item[-] equip $F$ with a $G\times H$-action
\begin{eqnarray*}
&(g,h)(x,y,k)=(gx,hy,\tau_y(h) k\thh_x(g)\iv)&\\
&(\fa (g,h)\in G\times H,\ \ \fa (x,y,k)\in F),&
\end{eqnarray*}
\item[-] define a 2-cell $\kappa\co\al\ci\wp_X\tc\be\ci\wp_Y$ by
\[ \kappa_{(x,y,k)}=k, \]
\end{itemize}
then the diagram
\begin{equation}\label{Diag18_0}
\xy
(-10,7)*+{\frac{F}{G\times H}}="0";
(10,7)*+{\xg}="2";
(-10,-6)*+{\yh}="4";
(10,-6)*+{\zk}="6";
{\ar^(0.52){\frac{\wp_X}{\prg}} "0";"2"};
{\ar_{\frac{\wp_Y}{\prh}} "0";"4"};
{\ar^{\althh} "2";"6"};
{\ar_{\bet} "4";"6"};
{\ar@{=>}^{\kappa} (2,2);(-2,-2)};
\endxy
\end{equation}
gives a bipullback $($\cite[P.155]{JS}$)$ in $\Sbb$.

In particular if $\zk=\pte$, then $\xg\ov{\frac{\wp_X}{\prg}}{\lla}\frac{F}{G\ti H}\ov{\frac{\wp_Y}{\prh}}{\lra}\yh$ gives a biproduct of $\xg$ and $\yh$. Here $\pt$ denotes a set with one element, with the trivial group action.
\end{fact}
\begin{rem}\label{Rema2}
$\el$ sends $(\ref{Diag18_0})$ to the comma square $($\cite[Diagram 1.12]{Borceux}, \cite[Proposition 1.26]{Groth}$)$ of categories
\[
\xy
(-10,7)*+{\el(\althh)/\el(\bet)}="0";
(10,7)*+{\el(\xg)}="2";
(-10,-6)*+{\el(\yh)}="4";
(10,-6)*+{\el(\zk)}="6";
{\ar^{} "0";"2"};
{\ar_{} "0";"4"};
{\ar^{\el(\althh)} "2";"6"};
{\ar_{\el(\bet)} "4";"6"};
{\ar@{=>}^{} (2,2);(-2,-2)};
\endxy.
\]
\end{rem}

\begin{prop}\label{Bieq}
For any $\Gcal\in\finGpd^0$, there is $\xg\in\Sbb^0$ admitting an equivalence $\el(\xg)\simeq\Gcal$. In particular, $\el\co\Sbb\to\finGpd$ is a biequivalence as in \cite[section 2.2]{Leinster}, by Corollary \ref{Cor2Ftr}.
\end{prop}
\begin{proof}
By Remark \ref{RemCCCCC}, we may assume $\Gcal$ is connected. Take any $x\in\Ob(\Gcal)$ and put $\Aut_{\Gcal}(x)=\Gcal(x,x)$. Then the 0-cell $\frac{\{ x\}}{\Aut_{\Gcal}(x)}\in\Sbb^0$ is sent by $\el$ to a full subcategory of $\Gcal$.

For any $x\ppr\in\Ob(\Gcal)$, there is an isomorphism $x\to x\ppr$ since $\Gcal$ is connected groupoid. This means $\el(\frac{\{ x\}}{\Aut_{\Gcal}(x)})$ is dense in (and hence equivalent to) $\Gcal$.
\end{proof}

\begin{dfn}\label{DefC}
Category $\Csc=\Sbb/\text{2-cells}$ is defined as follows.
\begin{itemize}
\item[{\rm (i)}] $\Ob(\C)=\Sbb^0$, namely, an object in $\Csc$ is a 0-cell in $\Sbb$.
\item[{\rm (ii)}] For any pair of objects $\xg,\yh\in\Ob(\Csc)$, define as
\[ \C(\xg,\yh)=\cl(\Sbb(\xg,\yh)). \]
The isomorphism class of $\althh\in\Ob(\Sbb(\xg,\yh))=\Sbb^1(\xg,\yh)$ is denoted by $\und{(\althh)}$, or abbreviately by $\und{\al}$.
\end{itemize}
\end{dfn}

\begin{rem}\label{RemImmed}
For any 1-cell $\al\co\xg\to\yh$ in $\Sbb$, the following are equivalent.
\begin{enumerate}
\item $\und{\al}$ is an isomorphism in $\Csc$.
\item $\al$ is an {\it equivalence} in $\Sbb$. Namely, there is a 1-cell $\be\co\yh\to\xg$ and 2-cells $\rho\co \be\ci\al\tc\id$, $\lam\co \al\ci\be\tc\id$.
\end{enumerate}
(As stated in Remark \ref{RemAbb}, 1-cells are abbreviated by $\al,\be$.)
\end{rem}

\begin{dfn}\label{DefNWP}
While biproducts and bicoproducts in $\Sbb$ yield products and coproducts in $\C$, remark that the image of a bipullback in $\Sbb$ becomes only a weak pullback in $\Csc$.

A weak pullback in $\Csc$
\[
\xy
(-8,6)*+{\wl}="0";
(8,6)*+{\yh}="2";
(-8,-6)*+{\xg}="4";
(8,-6)*+{\zk}="6";
{\ar^{\und{\delta}} "0";"2"};
{\ar_{\und{\gamma}} "0";"4"};
{\ar^{\und{\be}} "2";"6"};
{\ar_{\und{\al}} "4";"6"};
%{\ar@{=>}^{\kappa} (-2,0);(2,0)};
{\ar@{}|\circlearrowright "0";"6"};
\endxy
\]
is called a {\it natural weak pullback} (of $\und{\al}$ and $\und{\be}$) if it comes from some bipullback
\[
\xy
(-8,6)*+{\wl}="0";
(8,6)*+{\yh}="2";
(-8,-6)*+{\xg}="4";
(8,-6)*+{\zk}="6";
{\ar^{\delta} "0";"2"};
{\ar_{\gamma} "0";"4"};
{\ar^{\be} "2";"6"};
{\ar_{\al} "4";"6"};
{\ar@{=>}^{\kappa} (2,2);(-2,-2)};
\endxy
\]
in $\Sbb$.
We write as 
\begin{equation}\label{RCoeffAdd2}
\xy
(-8,6)*+{\wl}="0";
(8,6)*+{\yh}="2";
(-8,-6)*+{\xg}="4";
(8,-6)*+{\zk}="6";
(0,0)*+{\nwp}="10";
{\ar^{\und{\delta}} "0";"2"};
{\ar_{\und{\gamma}} "0";"4"};
{\ar^{\und{\be}} "2";"6"};
{\ar_{\und{\al}} "4";"6"};
\endxy
\end{equation}
to indicate it is a natural weak pullback. In this way, we can give $\C$ a class of natural weak pullbacks.
\end{dfn}

\begin{cor}\label{Property_of_C}
The category $\Csc$ has the following.
\begin{enumerate}
\item Initial object $\ems=\frac{\ems}{e}$, and terminal object $\pte$.
\item Any binary product and binary coproduct.
\item A class of natural weak pullbacks.
\end{enumerate}
In particular by {\rm (1),(2)}, $\C$ has any finite product and any finite coproduct. 
\end{cor}

\begin{rem}\label{Rem112}
Let $\Ksc$ be any category.
To give a functor $F\co\Csc\to\Ksc$ is equivalent to give a strict 2-functor $F\co\Sbb\to\Ksc$, where $\Ksc$ is regarded as a 2-category equipped only with identity 2-cells.
\end{rem}

\begin{rem}\label{RemCandC}
If one define a category $\C\ppr=\finGpd/\text{2-cells}$ in the same manner, then there is an equivalence $\C\ov{\simeq}{\hookrightarrow}\C\ppr$ which preserves the class of natural weak pullbacks, by Remark \ref{Rema2} and Proposition \ref{Bieq}.
\end{rem}

We define the notions of a (semi-)Mackey functor on $\Sbb$ and on $\Csc$, which are the same by Remark \ref{Rem112}.

\begin{dfn}\label{DefSemiMackC}\label{DefSemiMackS}
A {\it semi-Mackey functor} (respectively, a {\it $k$-linear Mackey functor}) $M=(M_{!},M^{\ast})$ on $\Csc$ is a pair of a contravariant functor $M^{\ast}\co\Csc\to\Sett$ (resp. $M^{\ast}\co\Csc\to\kMod$) and a covariant functor $M_{!}\co\Csc\to\Sett$ (resp. $M_{!}\co\Csc\to\kMod$) which satisfies the following.

($\Sett$ denotes the category of sets, where morphisms are maps of sets. $\kMod$ denotes the category of $k$-modules, where morphisms are $k$-linear maps.)
\begin{enumerate}
\item[{\rm (0)}] $M^{\ast}(\xg)=M_{!}(\xg)$ for any object $\xg\in\Ob(\Csc)$. We denote this simply by $M(\xg)$.
\item[{\rm (1)}] [Additivity] For any pair of objects $\xg$ and $\yh$ in $\Csc$, if we take their coproduct
\[ \xg\ov{\und{\ups_X}}{\lra}\xg\am\yh\ov{\und{\ups_Y}}{\lla}\yh \]
in $\Csc$, then the natural map
\[ (M^{\ast}(\und{\ups_X}),M^{\ast}(\und{\ups_Y}))\co M(\xg\am\yh)\to M(\xg)\times M(\yh) \]
is isomorphism. Also, $M(\emptyset)$ is trivial (i.e., the terminal object).
\item[{\rm (2)}] [Mackey condition] For any natural weak pullback
\[
\xy
(-10,7)*+{\wl}="0";
(10,7)*+{\yh}="2";
(-10,-7)*+{\xg}="4";
(10,-7)*+{\zk}="6";
(0,0)*+{\nwp}="10";
{\ar^{\und\delta} "0";"2"};
{\ar_{\und\gamma} "0";"4"};
{\ar^{\und\be} "2";"6"};
{\ar_{\und\al} "4";"6"};
%{\ar@{=>}^{\ep} (-2,0);(2,0)};
\endxy
\]
in $\Csc$, the following diagram in $\Sett$ (resp. $\kMod$) becomes commutative.
\[
\xy
(-12,7)*+{M(\wl)}="0";
(12,7)*+{M(\yh)}="2";
(-12,-7)*+{M(\xg)}="4";
(12,-7)*+{M(\zk)}="6";
{\ar_{M^{\ast}(\und\delta)} "2";"0"};
{\ar_{M_{!}(\und\gamma)} "0";"4"};
{\ar^{M_{!}(\und\be)} "2";"6"};
{\ar^{M^{\ast}(\und\al)} "6";"4"};
{\ar@{}|\circlearrowright "0";"6"};
\endxy
\]
\end{enumerate}

We can alternatively define a semi-Mackey functor by using $\Sbb$. In the following, when we speak of a 2-functor from $\Sbb$ to $\Sett$ or to $\kMod$, we regard it as a 2-category equipped only with identity 2-cells. (See Remark \ref{Rem112}.)

\medskip

A {\it semi-Mackey functor} $M=(M_{!},M^{\ast})$ on $\Sbb$ is a pair of a contravariant 2-functor $M^{\ast}\co\Sbb\to\Sett$ and a covariant 2-functor $M_{!}\co\Sbb\to\Sett$ which satisfies the following.
\begin{enumerate}
\item[{\rm (0)}] $M^{\ast}(\xg)=M_{!}(\xg)$ for any 0-cell $\xg\in\Sbb^0$. We denote this simply by $M(\xg)$.
\item[{\rm (1)}] [Additivity] For any pair of 0-cells $\xg$ and $\yh$ in $\Sbb$, if we take their bicoproduct
\[ \xg\ov{\ups_X}{\lra}\xg\am\yh\ov{\ups_Y}{\lla}\yh \]
in $\Sbb$, then the natural map
\begin{equation}\label{RCoeffAdd1}
(M^{\ast}(\ups_X),M^{\ast}(\ups_Y))\co M(\xg\am\yh)\to M(\xg)\times M(\yh)
\end{equation}
is isomorphism. Also, $M(\emptyset)$ is trivial.
\item[{\rm (2)}] [Mackey condition] For any bipullback
%\begin{equation}\label{RCoeffAdd2}
\[
\xy
(-10,7)*+{\wl}="0";
(10,7)*+{\yh}="2";
(-10,-7)*+{\xg}="4";
(10,-7)*+{\zk}="6";
{\ar^{\delta} "0";"2"};
{\ar_{\gamma} "0";"4"};
{\ar^{\be} "2";"6"};
{\ar_{\al} "4";"6"};
{\ar@{=>}^{\kappa} (2,2);(-2,-2)};
\endxy
\]
%\end{equation}
in $\Sbb$, the following diagram in $\Sett$ (resp. $\kMod$) becomes commutative.
\begin{equation}\label{RCoeffAdd3}
\xy
(-12,7)*+{M(\wl)}="0";
(12,7)*+{M(\yh)}="2";
(-12,-7)*+{M(\xg)}="4";
(12,-7)*+{M(\zk)}="6";
{\ar_{M^{\ast}(\delta)} "2";"0"};
{\ar_{M_{!}(\gamma)} "0";"4"};
{\ar^{M_{!}(\be)} "2";"6"};
{\ar^{M^{\ast}(\al)} "6";"4"};
{\ar@{}|\circlearrowright "0";"6"};
\endxy
\end{equation}
\end{enumerate}
This is just a paraphrase of the definition using $\Csc$. %We thus do not distinguish these two notions.
With this view, for any morphism $\und{\al}$ in $\Csc$, we write $M^{\ast}(\und{\al})$ and $M_{!}(\und{\al})$ also as $M^{\ast}(\al)$ and $M_{!}(\al)$.
\end{dfn}

\begin{rem}
In \cite{N_BisetMackey}, we used the notation $(M\uas,M\sas)$ to denote a Mackey functor, where $M\uas$ is contravariant, and $M\sas$ is covariant. In this article, we prefer to use $(M_{!},M\uas)$, because of the following reason.

In analogy with the ordinary Mackey functor theory, a {\it Tambara biset functor} is expected to be defined as a triplet of functors $(M_{!},M\uas,M\sas)$, consisting of an additive Mackey functor $(M_{!},M\uas)$ and a multiplicative semi-Mackey functor $(M\uas,M\sas)$. The additive part $(M_{!},M\uas)$ is expected to become a {\it Green functor} on $\Sbb$. The aim of this article is to give a framework to deal with Green functors, which will serve to the study of this additive part.
\end{rem}

\begin{rem}
If $M=(M_{!},M\uas)$ is a semi-Mackey functor on $\C$, then each of $M\uas$ and $M_{!}$ becomes a functor to $\Mon$, as shown in \cite[Proposition 5.17]{N_BisetMackey}. Thus in the definition of a semi-Mackey functor, we may assume $M\uas,M_{!}$ are functors to $\Mon$, from the beginning. ($\Mon$ denotes the category of commutative monoids with units, where morphisms are monoid homomorphisms preserving units.)
\end{rem}

\begin{rem}
$\MackSk$ is a $k$-linear abelian category.
\end{rem}

\begin{ex}$($\cite[Example 5.4.1]{N_BisetMackey}$)$\label{Abig}
A semi-Mackey functor $\Abig$ on $\C$ is defined as follows.
\begin{enumerate}
\item For any $\xg\in\Ob(\C)$, the set $\Abig(\xg)$ is defined to be the set of isomorphism classes of the slice category $\C/\xg$. Coproducts and natural weak pullbacks in $\C$ give a commutative semi-ring structure on $\Abig(\xg)$.
We denote the isomorphism class of an object $(\ak\ov{\und{\afr}}{\to}\xg)$ by $[\ak\ov{\und{\afr}}{\to}\xg]$.
\item Let $\und{\al}\in\C(\xg,\yh)$ be any morphism.
\begin{itemize}
\item[{\rm (i)}] $\Afr_{\mathrm{big}\, !}(\al)=\und{\al}\ci-\co \Abig(\xg)\to\Abig(\yh)$ is defined by the composition with $\und{\al}$.
\item[{\rm (ii)}] $\Abig\uas(\al)=\xg\un{\yh}{\ti}-\co \Abig(\yh)\to\Abig(\xg)$ is defined by the natural weak pullback by $\und{\al}$.
\end{itemize}
\end{enumerate}
If we compose with the additive completion functor $K_0\co\Mon\to\Ab$, we obtain a $\Zbb$-Mackey functor
\[ \Obig=(\Om_{\mathrm{big\, !}},\Obig\uas)=(K_0\ci\Afr_{\mathrm{big}\,!},K_0\ci\Abig\uas)\in\Ob(\Mack^{\Zbb}(\C)) \]
which we call the {\it bigger Burnside functor}. Furthermore, by composing the coefficient change functor $k\ot_{\Zbb}-\co\Ab\to\kMod$, we obtain a $k$-linear Mackey functor
\[ \Obigk=(\Omega^k_{\mathrm{big}\,!},\Obig^{k\ast})=((k\ot_{\Zbb}-)\ci\Omega_{\mathrm{big}\,!},(k\ot_{\Zbb}-)\ci\Obig\uas)\in\Ob(\Mack^k(\C)). \]
\end{ex}

\begin{rem}\label{PropIndEquiv2}
Let $\xg\in\Ob(\Csc)$ be any object. For any $x\in X$, if we denote its stabilizer by $G_x$ and the orbit by $Gx$, then there is an isomorphism in $\C$
\[ \frac{Gx}{G}\ov{\cong}{\lra}\frac{\pt}{G_x} \]
by Fact \ref{PropIndEquiv} and Remark \ref{RemImmed}.
If we take a set of representatives $x_1,\ldots,x_s\in X$ of $G$-orbits, thus we obtain an isomorphism
\[ \xg=\frac{Gx_1}{G}\am\cdots\am\frac{Gx_s}{G}\simeq\frac{\pt}{G_{x_1}}\am\cdots\am\frac{\pt}{G_{x_s}}. \]
\end{rem}

\begin{dfn}\label{DefSemiMackMorph}
Let $M$ and $N$ be semi-Mackey functors on $\Sbb$. A {\it morphism} $\varphi\co M\to N$ of semi-Mackey functors is a family of maps
\[ \varphi=\{ \varphi_{\xg}\co M(\xg)\to N(\xg) \}_{\xg\in\Sbb^0} \]
compatible with contravariant and covariant parts. Namely, it gives natural transformations
\[ \varphi\co M^{\ast}\tc N^{\ast}\ \ \ \text{and}\ \ \ \varphi\co M_{!}\tc N_{!}. \]
With the usual composition of natural transformations, we obtain the category of semi-Mackey functors denoted by $\SMackS$ or $\SMackC$.

Similarly, a {\it morphism} $\varphi\co M\to N$ of $k$-linear Mackey functors is a family $\varphi=\{ \varphi_{\xg}\}_{\xg\in\Sbb^0}$ of $k$-homomorphisms compatible with contravariant and covariant parts.
We denote the category of $k$-linear Mackey functors by $\MackSk$, or by $\MackCk$.
\end{dfn}

\begin{rem}
We can define Mackey functors on $\finGpd$ in the same way. This kind of generalization of a Mackey functor onto higher categories can be also found in \cite{Barwick}.
By Remark \ref{RemCandC}, the category $\SMack(\finGpd)=\SMack(\C\ppr)$ becomes equivalent to $\SMackC$, and $\Mack^k(\finGpd)=\Mack^k(\C\ppr)$ becomes equivalent to $\MackCk$.
\end{rem}

\begin{dfn}\label{DefStabsurj}$($\cite[Definition 4.1.1]{N_BisetMackey}$)$
A 1-cell $\althh\co\xg\to\yh$ in $\Sbb$ is called {\it stab-surjective}, if the following conditions are satisfied.
\begin{itemize}
\item[{\rm (i)}] $Y=H\al(X)$ holds.
\item[{\rm (ii)}] If $x,x\ppr\in X$ and $h,h\ppr\in H$ satisfy $h\al(x)=h\ppr\al(x\ppr)$, then there exists $g\in G$ which satisfies $x\ppr=gx$ and $h=h\ppr\thh_x(g)$.
\end{itemize}
Stab-surjectivity is stable under isomorphisms (by 2-cells) of 1-cells, and thus we can speak of the stab-surjectivity of a morphism $\und{\al}$ in $\C$.
\end{dfn}

\begin{rem}\label{RemStabsurj}$($\cite[section 4.1]{N_BisetMackey}$)$
The following holds for the stab-surjectivity.
\begin{enumerate}
\item Any isomorphism in $\C$ is stab-surjective.
\item Stab-surjectivity is closed under compositions. Namely, if $\xg\ov{\und{\al}}{\lra}\yh\ov{\und{\be}}{\lra}\zk$ is a sequence of stab-surjective morphisms, then so is $\und{\be}\ci\und{\al}$.
\item Stab-surjectivity is stable under natural weak pullbacks. Namely, if $(\ref{RCoeffAdd2})$ is a natural weak pullback and if $\und{\be}$ is stab-surjective, then so is $\und{\gamma}$.
\end{enumerate}
\end{rem}

\begin{dfn}\label{DefSIm}$($\cite[Definition 4.2.1]{N_BisetMackey}$)$
Let $\althh\co\xg\to\yh$ be any 1-cell in $\Sbb$.
\begin{enumerate}
\item Define $\SIm(\al)=\SIm(\althh)\in\Ob(\Hs)$ by $\SIm(\al)=(H\ti X)/\sim$, where the relation $\sim$ is defined as follows.
\begin{itemize}
\item[-] $(\eta,x),(\eta\ppr,x\ppr)\in H\ti X$ are equivalent if there exists $g\in G$ satisfying
\[ x\ppr=gx\ \ \ \text{and}\ \ \ \eta=\eta\ppr\thh_x(g). \]
\end{itemize}
We denote the equivalence class of $(\eta,x)$ by $[\eta,x]$. The $H$-action on $\SIm(\al)$ is given by $h[\eta, x]=[h\eta,x]$ for any $h\in H$.
We call $\SIm(\al)$ the {\it stabilizerwise image} of $\al=\althh$.
\item Define a map $\ups_{\al}\co X\to\SIm(\al)$ by
\[ \ups_{\al}(x)=[e,x]\quad(\fa x\in X). \]
Then $\frac{\ups_{\al}}{\thh}\co\xg\to \frac{\SIm(\al)}{H}$ is a stab-surjective 1-cell.
\end{enumerate}
\end{dfn}

\begin{dfn}\label{DefSIm2}$($\cite[Proposition 4.2.6]{N_BisetMackey}$)$
For any 1-cell $\althh\co\xg\to\yh$, we have a commutative diagram of 1-cells
\[
\xy
(-20,0)*+{\xg}="0";
(0,8)*+{\frac{\SIm(\al)}{H}}="2";
(0,-6)*+{}="3";
(20,0)*+{\frac{Y}{H}}="4";
{\ar^(0.46){\frac{\ups_{\al}}{\thh}} "0";"2"};
{\ar^(0.54){\frac{\wt{\al}}{H}} "2";"4"};
{\ar@/_0.8pc/_{\althh} "0";"4"};
{\ar@{}|\circlearrowright "2";"3"};
\endxy
\]
where $\wt{\al}$ is defined by $\wt{\al}([\eta,x])=\eta\al(x)$. We call this the {\it $\SIm$-factorization} of $\al$.
If $\und{\al}$ factorizes also as 
\[
\xy
(-20,0)*+{\xg}="0";
(0,8)*+{\frac{S}{H}}="2";
(0,-6)*+{}="3";
(20,0)*+{\frac{Y}{H}}="4";
{\ar^(0.46){\und{\ups}\ppr} "0";"2"};
{\ar^(0.54){\und{a}\ppr} "2";"4"};
{\ar@/_0.8pc/_{\und{\al}} "0";"4"};
{\ar@{}|\circlearrowright "2";"3"};
\endxy
\]
with stab-surjective $\ups\ppr$ and equivariant $a\ppr$, then there exists an $H$-equivariant equivalence $\frac{\om}{H}\co\frac{\SIm(\al)}{H}\ov{\simeq}{\lra}\frac{S}{H}$ in $\Sbb$ satisfying $\und{\ups}\ppr=\und{\om}\ci\und{\ups_{\al}}$ and $\und{\wt{\al}}=\und{a}\ppr\ci\und{\om}$. This ensures the uniqueness of the $\SIm$-factorization, up to isomorphisms in $\C$.
\end{dfn}

\begin{dfn}\label{DefDeflMack}$($\cite[Definition 5.3.1]{N_BisetMackey}$)$
A semi-Mackey functor (respectively, $k$-linear Mackey functor) $M$ on $\C$ is called {\it deflative} if it satisfies
\[ M_{!}(\al)\ci M^{\ast}(\al)=\id_{M(\yh)} \]
for any stab-surjective morphism $\und{\al}\in\C(\xg,\yh)$.
The full subcategory of deflative semi-Mackey functors is denoted by $\SMack_{\dfl}(\Csc)\subseteq\SMackC$. 
Similarly, the full subcategory of deflative $k$-linear Mackey functors is denoted by $\MackdCk\subseteq\MackCk$.
\end{dfn}

\begin{dfn}\label{DefOrdBurn}
Let $\xg$ be any object in $\C$. The set of isomorphism classes of the slice category $\Gs/X$ is equipped with a commutative semi-ring structure, whose addition and multiplication are induced from coproducts and fibered products of finite $G$-sets. We denote this semi-ring by $\Afr_G(X)$. By taking its additive completion, we obtain the ordinary Burnside ring
\[ \Om_G(X)=K_0(\Afr_G(X)). \]
Tensoring with $k$, we define $\Omk_G(X)$ by $\Omk_G(X)=k\ot_{\Zbb}\Om_G(X)$.
\end{dfn}

The following homomorphisms have been obtained in \cite{N_BisetMackey}.
\begin{dfn}\label{DefBurntoBig}(\cite[Proposition 5.4.10]{N_BisetMackey})
For any object $\xg$ in $\C$, we have the following.
\begin{enumerate}
\item A ring homomorphism $\ibf_{\xg}\co\Omk_G(X)\to\Obigk(\xg)$ is obtained by extending the map
\[ \Afr_G(X)\to\Obigk(\xg)\ ;\ [A\ov{p}{\lra}X]\mapsto [\frac{A}{G}\ov{\und{(\frac{p}{G})}}{\lra}\xg] \]
by $k$-linearity.
\item A ring homomorphism $\pbf_{\xg}\co\Obigk(\xg)\to\Omk_G(X)$ is obtained by extending the map induced from $\SIm$-factorizations
\[ \Abig(\xg)\to\Omk_G(X)\ ;\ [\akaxgu]\to [\SIm (\afr)\ov{\wt{\afr}}{\to}X], \]
by $k$-linearity.
\end{enumerate}
These homomorphisms satisfy $\pbf_{\xg}\ci\ibf_{\xg}=\id_{\Omk_G(X)}$. In particular, $\pbf_{\xg}$ is surjective.
\end{dfn}

\begin{dfn}\label{RemStrOrdBurn}
$\Omk\in\Ob(\MackdCk)$ is given by the following.
\begin{enumerate}
\item To any object $\xg$ in $\C$, associate $\Omk(\xg)=\Omk_{G}(X)$.
\item Let $\und{\al}\co \xg\to\yh$ be any morphism in $\C$.
\begin{itemize}
\item[{\rm (i)}] $\Omk_{!}(\al)\co \Omk(\xg)\to\Omk(\yh)$ is defined to be the composition of
\[ \Omk_G(X)\ov{\ibf_{\xg}}{\lra}\Obigk(\xg)\ov{\Omega^k_{\mathrm{big}\,!}(\al)}{\lra}\Obigk(\yh) \ov{\pbf_{\yh}}{\lra}\Omk_H(Y). \]
\item[{\rm (ii)}] $\Om^{k\ast}(\al)\co \Omk(\yh)\to\Omk(\xg)$ is defined to be the composition of
\[ \Omk_H(Y)\ov{\ibf_{\xg}}{\lra}\Obigk(\yh)\ov{\Obig^{k\ast}(\al)}{\lra}\Obigk(\xg)\ov{\pbf_{\xg}}{\lra}\Omk_G(X). \]
\end{itemize}
\end{enumerate}
The fact that $(\Omk_{!},\Om^{k\ast})$ indeed belongs to $\MackCk$, can be checked according to the definition. Or, it also follows from the surjectivity of $\pbf_{\xg}$ and the compatibilities obtained in the proof of the next proposition. Deflativity is obvious from the definition.
\end{dfn}

\begin{prop}\label{PropStrOrdBurn}
The natural surjections
\[ \pbf_{\xg}\co \Obigk(\xg)\to\Omk_G(X)\ \ \quad (\fa\xg\in\Ob(\C)) \]
form a morphism of Mackey functors $\pbf\co\Obigk\to\Omk$ on $\C$. This is an epimorphism in $\MackCk$.
\end{prop}
\begin{proof}
It suffices to show that $\pbf$ is natural with respect to the covariant and the contravariant parts.
Epimorphicity of $\pbf$ is obvious from the surjectivity of $\pbf_{\xg}\ (\fa\xg\in\Ob(\C))$.

Let $\und{\al}\co\xg\to\yh$ be any morphism.

\smallskip

\noindent {\bf [Compatibility with the covariant parts]}

For any morphism $\und{\afr}\co\ak\to\xg$, take the $\SIm$-factorization of $\afr$
\[
\xy
(-20,0)*+{\ak}="0";
(0,8)*+{\frac{\SIm(\afr)}{G}}="2";
(0,-6)*+{}="3";
(20,0)*+{\xg}="4";
{\ar^(0.46){\ups_{\afr}} "0";"2"};
{\ar^(0.54){\frac{\wt{\afr}}{G}} "2";"4"};
{\ar@/_0.8pc/_{\afr} "0";"4"};
{\ar@{}|\circlearrowright "2";"3"};
\endxy.
\]
Then by definition we have $\ibf_{\xg}\ci\pbf_{\xg}([\akaxgu])=[\frac{\SIm(\afr)}{G}\ov{\und{(\frac{\wt{\afr}}{G})}}{\lra}\xg]$.
By the stab-surjectivity of $\ups_{\afr}$, we obtain an isomorphism of $H$-sets
\[ \SIm(\al\ci\afr)=\SIm(\al\ci\frac{\wt{\afr}}{G}\ci\ups_{\afr})\cong\SIm(\al\ci\frac{\wt{\afr}}{G}), \]
which means
\[ \pbf_{\yh}\ci(\Omega^k_{\mathrm{big}\,!}(\al))([\akaxgu])=\pbf_{\yh}\ci(\Omega^k_{\mathrm{big}\,!}(\al))([\frac{\SIm(\afr)}{G}\ov{\und{(\frac{\wt{\afr}}{G})}}{\lra}\xg]). \]
This implies the commutativity of the following diagram.
\[
\xy
(-20,18)*+{\Obigk(\xg)}="0";
(20,18)*+{\Omk(\xg)}="2";
(2,6)*+{\Obigk(\xg)}="4";
(2,-6)*+{\Obigk(\yh)}="6";
(-20,-18)*+{\Obigk(\yh)}="8";
(20,-18)*+{\Omk(\yh)}="10";
(-13,0)*+{_{\circlearrowright}}="11";
(12,0)*+{_{\circlearrowright}}="12";
{\ar^{\pbf_{\xg}} "0";"2"};
{\ar_{\ibf_{\xg}} "2";"4"};
{\ar_{\Omega^k_{\mathrm{big}\, !}(\al)} "4";"6"};
{\ar_{\Omega^k_{\mathrm{big}\, !}(\al)} "0";"8"};
{\ar^{\Omk_{!}(\al)} "2";"10"};
{\ar_{\pbf_{\yh}} "6";"10"};
{\ar_{\pbf_{\yh}} "8";"10"};
\endxy
\]

\medskip

\noindent {\bf [Compatibility with the contravariant parts]}

For any morphism $\und{\bfr}\co\bl\to\yh$, take the $\SIm$-factorization of $\bfr$
\[
\xy
(-20,0)*+{\bl}="0";
(0,8)*+{\frac{\SIm(\bfr)}{H}}="2";
(0,-6)*+{}="3";
(20,0)*+{\yh}="4";
{\ar^(0.46){\ups_{\bfr}} "0";"2"};
{\ar^(0.54){\frac{\wt{\bfr}}{H}} "2";"4"};
{\ar@/_0.8pc/_{\bfr} "0";"4"};
{\ar@{}|\circlearrowright "2";"3"};
\endxy.
\]
If we take natural weak pullbacks by $\und{\al}$ as
\[
\xy
(-8,14)*+{\blp}="0";
(8,14)*+{\bl}="2";
(-8,0)*+{\frac{S}{H\ppr}}="4";
(8,0)*+{\frac{\SIm(\bfr)}{H}}="6";
(21,0)*+{}="7";
(-8,-14)*+{\xg}="8";
(8,-14)*+{\yh}="10";
(0,7)*+{\nwp}="20";
(0,-7)*+{\nwp}="21";
{\ar^{\und{\al}\pprr} "0";"2"};
{\ar_{\und{\ups}\ppr} "0";"4"};
{\ar^{\und{\ups}_{\bfr}} "2";"6"};
{\ar_{\und{\al}\ppr} "4";"6"};
{\ar_{\und{\sfr}} "4";"8"};
{\ar^{\und{(\frac{\wt{\bfr}}{H})}} "6";"10"};
{\ar_{\und{\al}} "8";"10"};
{\ar@/^2.8pc/^{\und{\bfr}} "2";"10"};
{\ar@{}|\circlearrowright "6";"7"};
\endxy,
\]
then by definition, we have
\[ \Obig^{k\ast}(\al)([\frac{\SIm(\bfr)}{H}\ov{\und{(\frac{\wt{\bfr}}{H})}}{\lra}\yh])=[\frac{S}{H\ppr}\ov{\und{\sfr}}{\lra}\xg] \]
and
\[ \Obig^{k\ast}(\al)([\blbyhu])=[\frac{B\ppr}{L\ppr}\ov{\und{\sfr\ci\ups\ppr}}{\lra}\xg]. \]
Since $\ups\ppr$ is stab-surjective by Remark \ref{RemStabsurj}, there is an isomorphism of $G$-sets
$\SIm(\sfr\ci\ups\ppr)\cong\SIm(\sfr)$,
and thus we have
\[ \pbf_{\xg}\Obig^{k\ast}([\blbyhu])=\pbf_{\xg}\Obig^{k\ast}([\frac{\SIm(\bfr)}{H}\ov{\und{(\frac{\wt{\bfr}}{H})}}{\lra}\yh]), \]
which implies the commutativity of the following diagram.
\[
\xy
(-20,18)*+{\Obigk(\yh)}="0";
(20,18)*+{\Omk(\yh)}="2";
(2,6)*+{\Obigk(\yh)}="4";
(2,-6)*+{\Obigk(\xg)}="6";
(-20,-18)*+{\Obigk(\xg)}="8";
(20,-18)*+{\Omk(\xg)}="10";
(-10,0)*+{_{\circlearrowright}}="11";
(16,0)*+{_{\circlearrowright}}="12";
{\ar^{\pbf_{\yh}} "0";"2"};
{\ar^{\ibf_{\yh}} "2";"4"};
{\ar^{\Obig^{k\ast}(\al)} "4";"6"};
{\ar_{\Obig^{k\ast}(\al)} "0";"8"};
{\ar^{\Om^{k\ast}(\al)} "2";"10"};
{\ar^{\pbf_{\xg}} "6";"10"};
{\ar_{\pbf_{\xg}} "8";"10"};
\endxy
\]
\end{proof}

\begin{prop}\label{PropHHom}
For any deflative Mackey functor $N$ on $\C$, the following holds.
\begin{enumerate}
\item For any $f\in\MackCk(\Obigk,N)$, we have
\[ f_{\xg}\ci\ibf_{\xg}\ci\pbf_{\xg}=f_{\xg} \]
for any $\xg\in\Ob(\C)$.
\item $\pbf$ induces a bijection
\begin{equation}\label{HHom2}
-\ci\pbf\co\MackdCk(\Omk,N)\ov{\cong}{\lra}\MackCk(\Obigk,N).
\end{equation}
\end{enumerate}
\end{prop}
\begin{proof}
{\rm (1)} It suffices to show
\[ f_{\xg}\ci\ibf_{\xg}\ci\pbf_{\xg}([\akaxgu])=f_{\xg}([\akaxgu]) \]
for any $[\akaxgu]\in\Abig(\xg)$. If we take the $\SIm$-factorization $\afr=\frac{\wt{\afr}}{G}\ci\ups_{\afr}$,
\[ N_{!}(\afr)N\uas(\afr)=N_{!}(\frac{\wt{\afr}}{G})N\uas(\frac{\wt{\afr}}{G}) \]
follows from the deflativity of $N$. Since
\begin{eqnarray*}
&{[}\akaxgu{]}=\Omega^k_{\mathrm{big}\,!}(\afr)\Obig^{k\ast}(\afr)([\xgixg]),&\\
&{[}\frac{\SIm(\afr)}{G}\ov{\und{(\frac{\wt{\afr}}{G})}}{\to}\xg{]}=\Omega^k_{\mathrm{big}\,!}(\frac{\wt{\afr}}{G})\Obig^{k\ast}(\frac{\wt{\afr}}{G})([\xgixg])&
\end{eqnarray*}
hold in $\Obigk(\xg)$, we have
\begin{eqnarray*}
f_{\xg}([\akaxgu])&=&N_{!}(\afr)N\uas(\afr)f_{\xg}([\xgixg])\\
&=&N_{!}(\frac{\wt{\afr}}{G})N\uas(\frac{\wt{\afr}}{G})f_{\xg}([\xgixg])\\
&=&f_{\xg}([\frac{\SIm(\afr)}{G}\ov{\und{(\frac{\wt{\afr}}{G}})}{\to}\xg])\ =\ f_{\xg}\ci\ibf_{\xg}\ci\pbf_{\xg}([\akaxgu]).
\end{eqnarray*}

{\rm (2)} Injectivity follows from the surjectivity of $\pbf_{\xg}$ for each $\xg\in\Ob(\C)$. For any $f\in\MackCk(\Obigk,N)$, if we put
\[ f\ppr_{\xg}=f_{\xg}\ci\ibf_{\xg}\quad(\fa\xg\in\Ob(\C)), \]
then $f_{\xg}=f\ppr_{\xg}\ci\pbf_{\xg}$ follows from {\rm (1)}. By the surjectivity of $\pbf_{\xg}$ and the naturality of $f$, we can show $f\ppr=\{ f\ppr_{\xg}\}_{\xg\in\Ob(\C)}$ is a morphism of Mackey functors. Thus $(\ref{HHom2})$ is surjective.
\end{proof}

A biset functor $B$ is defined to be an additive functor $B\co \Bcal\to\kMod$, from the {\it biset category} $\Bcal$ to $\kMod$.
The biset category which we deal with in this article is the following one. Throughout this article, a biset is always assumed to be finite.
\begin{dfn}$($\cite[Definitions 3.1.1, 3.1.6]{Bouc_biset}$)$
The category $\Bcal$ is defined as follows.
\begin{enumerate}
\item An object in $\Bcal$ is a finite group.
\item For objects $G,H$ in $\Bcal$, consider a set of the isomorphism classes of finite $H$-$G$-bisets. An isomorphism of $H$-$G$-bisets $U\ov{\cong}{\lra}U\ppr$ is a bijective map which preserves the left $H$-action and the right $G$-action.

This forms a commutative monoid with addition $\am$ and unit $\emptyset$, and thus we can take its additive completion $\Bcal(G,H)$. This is the set of morphisms from $G$ to $H$ in $\Bcal$.

An $H$-$G$-biset $U$ is written as ${}_HU_G$.
The composition of two consecutive bisets ${}_{H}U_G$ and ${}_{K}V_H$ is given by
\[ V\ti_HU=(V\ti U)/\sim, \]
where the equivalence relation is defined as
\begin{itemize}
\item[-] $(v,u),(v\ppr,u\ppr)\in V\ti U$ are equivalent if there exists $h\in H$ satisfying $v=v\ppr h$ and $u\ppr=hu$.
\end{itemize}
This defines the composition of morphisms in $\Bcal$, by additivity.
\end{enumerate}
By the abelian group structure on $\Bcal(G,H)$, category $\Bcal$ is preadditive.\end{dfn}
A biset functor is an additive functor $\Bcal\to\kMod$. We denote the category of biset functors by $\BFk=\Add(\Bcal,\kMod)$. Morphisms in $\BFk$ are natural transformations. 
The following has been shown in \cite{N_BisetMackey}.
\begin{fact}\label{ThmBF}(\cite[Theorem 6.3.11]{N_BisetMackey})
There is an equivalence of categories $\MackdCk\simeq\BFk$. This enables us to regard a biset functor as a deflative Mackey functor on $\C$.
\end{fact}

\section{From derivators to Mackey functors}

\begin{dfn}\label{DefPrederiv}(\cite[Definition 1.1]{Groth})
A {\it prederivator} on $\finCat$ is a 2-functor
\[ \Dbb\co\finCat^{\op}\to\CAT, \]
where $\CAT$ denotes the 2-category of categories.

More precisely, $\Dbb$ consists of the following correspondences.
\begin{enumerate}
\item[{\rm (0)}] To any finite category $\Ical$, a category $\Dbb(\Ical)$ is associated.
\item[{\rm (1)}] To any 1-cell $u\in\finCat^1(\Ical,\Jcal)$, a functor $\Dbb(u)=u\uas\co\Dbb(\Jcal)\to\Dbb(\Ical)$ is associated. This strictly preserves compositions and identities.
\item[{\rm (2)}] To any 2-cell $\lam\in\finCat^2(u,v)$, a natural transformation $\Dbb(\lam)=\lam\uas\co u\uas\tc v\uas$ is associated. This strictly preserves compositions and identities.
\end{enumerate}
For the detail, see \cite{Groth}.
\end{dfn}

\begin{rem}
The reason why we use a bit larger 2-category $\finCat$ than $\finGpd$ is, just because $\finGpd$ is not allowed as \lq {\it category of diagrams}' (\cite[Definition 1.12]{Groth}), while $\finCat$ is. The difference is that $\finCat$ contains any finite poset (viewed as a category) as 0-cell.
\end{rem}

\begin{ex}\label{ExYoneda}(\cite[Example 1.2]{Groth}, \cite[D\'{e}finitions 1.11, 1.23]{Cisinski})
Let $\A$ be any category. Then a prederivator $\Ycal_{\A}$ on $\finCat$ is defined as follows.
\begin{enumerate}
\item[{\rm (0)}] For any $\Ical\in\finCat^0$, let $\Ycal_{\A}(\Ical)=\A^{\Ical}$ be the functor category.
\item[{\rm (1)}] For any $u\in\finCat^1(\Ical,\Jcal)$, define $u\uas=-\ci u\co \A^{\Jcal}\to\A^{\Ical}$ by the composition with $u$.
\item[{\rm (2)}] For any $\lam\in\finCat^2(u,v)$, define $\lam\uas=-\ci \lam\co u\uas\tc v\uas$ by the horizontal composition with $\lam$.
\end{enumerate}
\end{ex}

\begin{dfn}\label{DefDeriv}(\cite[Definition 1.5]{Groth})
Let $\Dbb\co\finCat^{\op}\to\CAT$ be a prederivator. It is called a {\it derivator} if it satisfies the following condition. 
\begin{enumerate}
\item[{\rm (Der1)}] The empty category $\ems$ satisfies $\Dbb(\ems)\simeq\ebf$. For any $\Ical,\Jcal\in\finCat^0$,
\[ (i^{\ast},j^{\ast})\co\Dbb(\Ical\am\Jcal)\ov{\simeq}{\lra}\Dbb(\Ical)\ti\Dbb(\Jcal) \]
is an equivalence of categories, where
\[ \Ical\ov{i}{\lra}\Ical\am\Jcal\ov{j}{\lla}\Jcal \]
is the coproduct of categories $\Ical,\Jcal$.
\item[{\rm (Der2)}] Let $\Jcal\in\finCat^0$ be any 0-cell in $\finCat$. For any morphism $f$ in $\Dbb(\Jcal)$, it is an isomorphism in $\Dbb(\Jcal)$ if and only if $c_X^{\ast}(f)$ is isomorphism in $\Dbb(\ebf)$ for any $X\in\Ob(\Jcal)$. Here, $c_X\co \ebf\to\Jcal$ denotes the constant functor onto $X$.
\item[{\rm (Der3)}] For any $u\in\finCat^1(\Jcal,\Kcal)$, the functor $u^{\ast}\co\Dbb(\Kcal)\to\Dbb(\Jcal)$ has a left adjoint $u_{!}$ and a right adjoint $u_{\ast}$.
\item[{\rm (Der4)}] For any comma square
\begin{equation}\label{CommaSq}
\xy
(-7,6)*+{\Ical/\Jcal}="0";
(7,6)*+{\Ical}="2";
(-7,-6)*+{\Jcal}="4";
(7,-6)*+{\Kcal}="6";
{\ar^{p_{\Ical}} "0";"2"};
{\ar_{p_{\Jcal}} "0";"4"};
{\ar^{I} "2";"6"};
{\ar_{J} "4";"6"};
{\ar@{=>}^{\lam} (2,2);(-2,-2)};
\endxy
\end{equation}
in $\finCat$, the natural transformations
\[
\xy
(-9,7)*+{\Dbb(\Ical/\Jcal)}="0";
(9,7)*+{\Dbb(\Ical)}="2";
(-9,-7)*+{\Dbb(\Jcal)}="4";
(9,-7)*+{\Dbb(\Kcal)}="6";
{\ar_(0.4){p_{\Ical}^{\ast}} "2";"0"};
{\ar_{(p_{\Jcal})_{!}} "0";"4"};
{\ar^{I_{!}} "2";"6"};
{\ar^{J^{\ast}} "6";"4"};
{\ar@{=>}^{\lam_{!}} (-2,2);(2,-2)};
\endxy
\quad,\quad
\xy
(-9,7)*+{\Dbb(\Ical/\Jcal)}="0";
(9,7)*+{\Dbb(\Ical)}="2";
(-9,-7)*+{\Dbb(\Jcal)}="4";
(9,-7)*+{\Dbb(\Kcal)}="6";
{\ar^(0.56){(p_{\Ical})_{\ast}} "0";"2"};
{\ar^{p_{\Jcal}^{\ast}} "4";"0"};
{\ar_{I^{\ast}} "6";"2"};
{\ar_{J_{\ast}} "4";"6"};
{\ar@{=>}^{\lam_{\ast}} (2,-2);(-2,2)};
\endxy
\]
are isomorphisms. Here, $\lam_{!}$ and $\lam_{\ast}$ are defined by using the units and counits of the adjoint functors, by composing the following sequences of natural transformations, respectively.
\begin{eqnarray*}
&(p_{\Jcal})_{!}\ci p_{\Ical}^{\ast}\tc%
(p_{\Jcal})_{!}\ci p_{\Ical}^{\ast}\ci I^{\ast}\ci I_{!}\ov{(p_{\Jcal})_{!}\ci\lam^{\ast}\ci I_{!}}{\Longrightarrow}%
(p_{\Jcal})_{!}\ci p_{\Jcal}^{\ast}\ci J^{\ast}\ci I_{!}\tc%
J^{\ast}\ci I_{!},&\\
&I^{\ast}\ci J_{\ast}\tc%
(p_{\Ical})_{\ast}\ci p_{\Ical}^{\ast}\ci I^{\ast}\ci J_{\ast}\ov{(p_{\Ical})_{\ast}\ci\lam^{\ast}\ci J_{\ast}}{\Longrightarrow}%
(p_{\Ical})_{\ast}\ci p_{\Jcal}^{\ast}\ci J^{\ast}\ci J_{\ast}\tc%
(p_{\Ical})_{\ast}\ci p_{\Jcal}^{\ast}.&
\end{eqnarray*}
\end{enumerate}
\end{dfn}

\begin{rem}
In the above definition, {\rm (Der4)} is replaced by an equivalent (under the assumption of {\rm (Der1),(Der2),(Der3)}) condition from the original one (\cite[Proposition 1.26]{Groth}).
\end{rem}

\begin{fact}(\cite[Example 1.12]{Cisinski})
If a category $\A$ is finitely complete and finitely cocomplete, then
\[ \Ycal_{\A}\co\finCat^{\op}\to\CAT \]
becomes a derivator.
\end{fact}

Now we associate a semi-Mackey functor to any derivator on $\finCat$. In fact, we only need a 2-functor $\finGpd^{\op}\to\CAT$ satisfying the following conditions\footnote{These conditions are similar to {\rm (Der1),(Der3g),(Der4g)} for {\it d\'{e}rivateur faible \`{a} gauche} in \cite[D\'{e}finition 1.11]{Cisinski}. However, we remark that {\rm (iii)} is stronger than {\rm (Der4g)}, since we do not assume {\rm (Der2)}.} {\rm (i),(ii),(iii)}. Obviously, if we restrict a derivator $\Dbb\co\finCat^{\op}\to\CAT$ to $\finGpd\se\finCat$, these conditions are satisfied.
\begin{prop}\label{PropDerivtoMack}
Let $\Dbb\co\finGpd^{\op}\to\CAT$ be a strict 2-functor. Suppose $\Dbb$ satisfies the following properties.
\begin{itemize}
\item[{\rm (i)}] The empty category $\ems$ satisfies $\Dbb(\ems)\simeq\ebf$. For any $\Ical,\Jcal\in\finGpd^0$,
\[ (i^{\ast},j^{\ast})\co\Dbb(\Ical\am\Jcal)\ov{\simeq}{\lra}\Dbb(\Ical)\ti\Dbb(\Jcal) \]
is an equivalence of categories, where
\[ \Ical\ov{i}{\lra}\Ical\am\Jcal\ov{j}{\lla}\Jcal \]
is the coproduct of categories $\Ical,\Jcal$.
\item[{\rm (ii)}] For any $u\in\finGpd^1(\Jcal,\Kcal)$, the functor $u^{\ast}\co\Dbb(\Kcal)\to\Dbb(\Jcal)$ has a left adjoint $u_{!}$.
\item[{\rm (iii)}] For any comma square $(\ref{CommaSq})$ in $\finGpd$, the natural transformation $\lambda_{!}\co (p_{\Jcal})_{!}\ci p_{\Ical}\uas\tc J\uas\ci I_{!}$ is isomorphism\footnote{As the proof suggests, the existence of an isomorphism $(p_{\Jcal})_{!}\ci p_{\Ical}\uas\cong J\uas\ci I_{!}$ (not specifying the isomorphism used) is enough to show Proposition \ref{PropDerivtoMack}.}.
\end{itemize}

Then, we obtain a semi-Mackey functor $(M_{!},M^{\ast})$ on $\C$ in the following way.
\begin{itemize}
\item[-] For any $\xg\in\Ob(\C)$, put
\[ M(\xg)=\cl(\Dbb(\el(\xg))). \]
\item[-] For any $\und{\al}\in\C(\xg,\yh)$, put
\[ M^{\ast}(\und{\al})=\cl(\el(\al)^{\ast}),\quad M_{!}(\und{\al})=\cl(\el(\al)_{!}). \]
\end{itemize}
\end{prop}
\begin{proof}
If 1-cells satisfy $\und{\al}=\und{\al}\ppr$, then we have natural isomorphism $\el(\al)\cong\el(\al\ppr)$. It follows $\Dbb(\el(\al))\cong\Dbb(\el(\al\ppr))$, namely, $\el(\al)\uas\cong\el(\al\ppr)\uas$, and thus $\cl(\el(\al)\uas)=\cl(\el(\al\ppr)\uas)$. Thus $M\uas(\und{\al})$ is independent from a representative 1-cell $\al$. Obviously, $M\uas\co\C^{\op}\to\Sett$ becomes a functor.

From $\el(\al)\uas\cong\el(\al\ppr)\uas$, it also follows $\el(\al)_{!}\cong\el(\al\ppr)_{!}$. Similarly, this shows $M_{!}(\und{\al})$ is independent from a representative $\al$.
For any sequence of 1-cells $\xg\ov{\al}{\lra}\yh\ov{\be}{\lra}\zk$ in $\Sbb$,
\[ \el(\be\ci\al)_{!}\cong(\el(\be)\ci\el(\al))_{!}\cong\el(\be)_{!}\ci\el(\al)_{!} \]
follows from $\el(\be\ci\al)\uas=(\el(\be)\ci\el(\al))\uas=\el(\al)\uas\ci\el(\be)\uas$.
Thus we obtain $M_{!}(\und{\be}\ci\und{\al})=M_{!}(\und{\be})\ci M_{!}(\und{\al})$. We can also confirm $M_{!}(\id_{\xg})=\id$ for any $\xg\in\Ob(\C)$, in a similar way. Thus $M_{!}\co\C\to\Sett$ becomes a functor.

Then the additivity of $M\uas$ follows from condition {\rm (i)} and Remark \ref{RemCCCCC}. The Mackey condition follows from {\rm (iii)} and Remark \ref{Rema2}. Thus $(M_{!},M\uas)$ becomes a semi-Mackey functor as in Definition \ref{DefSemiMackC}.
\end{proof}

\begin{rem}
The same proof works if we replace left adjoint in {\rm (ii)} by right adjoint, and modify {\rm (iii)} accordingly.
\end{rem}

\begin{ex}\label{ExSemiBurn}
Let $\sett$ be the category of finite sets, whose morphisms are maps of sets.
Since it is finitely complete and finitely cocomplete, the represented prederivator $\Ycal_{\sett}\co\finCat^{\op}\to\CAT$ becomes a derivator. Applying Proposition \ref{PropDerivtoMack} to $\Ycal_{\sett}$, we obtain a semi-Mackey functor $M=(M_{!},M\uas)$, which satisfies the following.
\begin{enumerate}
\item For any $\xg\in\Sbb^0$, remark that there is an equivalence of categories
\[ \Ycal_{\sett}(\xg)=\Fun(\el(\xg),\sett)\ov{\simeq}{\lra}\Gs/X, \]
where the codomain denotes the slice category of $\Gs$ over $X$. (See \cite[P.24 Example {\rm (iv)}]{MM} for the case $X=\pt$.) Indeed, this equivalence is given in the following way.
\begin{itemize}
\item[{\rm (i)}] For any $(A\ov{p}{\to}X)\in\Ob(\Gs/X)$, define $F\in\Ob(\Fun(\el(\xg),\sett))$ by
\begin{itemize}
\item[-] for any object $x\in X=\Ob(\el(\xg))$, put $F(x)=p\iv(x)\in\Ob(\sett)$,
\item[-] for any morphism $g\in\el(\xg)(x,x\ppr)$, let $F(g)\co p\iv(x)\to p\iv(x\ppr)$ be the map given by the left action of $g$.
\end{itemize}
\item[{\rm (ii)}] Let $(A\ov{p}{\to}X),(A\ppr\ov{p\ppr}{\to}X)\in\Ob(\Gs/X)$ be any pair of objects, and let $F,F\ppr$ be the corresponding objects in $\Fun(\el(\xg),\sett)$ by {\rm (i)}. For any morphism $\phi\co(A\ov{p}{\to}X)\to(A\ppr\ov{p\ppr}{\to}X)$ in $\Gs/X$, a natural transformation $\varphi\co F\tc F\ppr$ is given by
\[ \varphi_x=\phi|_{p\iv(x)}\co p\iv(x)\to p^{\prime-1}(x) \]
for any $x\in X=\Ob(\el(\xg))$.
\end{itemize}
This induces an isomorphism $M(\xg)\cong \Afr_G(X)$. In particular if $X=\pt$, we have $M(\ptg)\cong\Afr_{G}(\pt)$.
\item If $f\co G\to H$ is a homomorphism between finite groups, then there is a natural isomorphism
\[
\xy
(-12,7)*+{\Ycal_{\sett}(\el(\ptg))}="0";
(12,7)*+{\Gs}="2";
(-12,-7)*+{\Ycal_{\sett}(\el(\pth))}="4";
(12,-7)*+{\Hs}="6";
{\ar^{\simeq} "0";"2"};
{\ar^{\el(\frac{\ttt}{f})\uas} "4";"0"};
{\ar_{f\uas} "6";"2"};
{\ar_{\simeq} "4";"6"};
{\ar@{}|{\cong} "0";"6"};
\endxy,
\]
where $f\uas\co\Hs\to\Gs$ is the functor given by the pullback of the action along $f$. Here $\ttt$ denotes the unique map.
\end{enumerate}
\end{ex}

\begin{rem}
For a general 1-cell $\althh$, we can give a functor $(\althh)\uas\co\Hs/Y\to\Gs/X$ as in \cite[Definition 2.2]{N_DerTam}.
We can also show that the semi-Mackey functor obtained in Example \ref{ExSemiBurn} is isomorphic to the semi-Burnside functor on $\Sbb$ (\cite[Definition 6.7]{N_DerTam}).
\end{rem}

\begin{ex}
In a similar way as in Proposition \ref{PropDerivtoMack}, we obtain the following. Let $\Cbb$ be the field of complex numbers, and let $\Cmod$ be the category of finite dimensional $\Cbb$-vector spaces, where morphisms are $\Cbb$-linear homomorphisms.
Let $\Dbb=\Ycal_{\Cmod}\co\finCat^{\op}\to\CAT$ be the derivator represented by $\Cmod$. 
Then, we obtain a $\Zbb$-Mackey functor $(M_{!},M^{\ast})$ on $\C$ in the following way.
\begin{enumerate}
\item For any $\xg\in\Ob(\C)$, put
\[ M(\xg)=K_0(\Dbb(\el(\xg))). \]
Remark that we have an equivalence $\Dbb(\el(\ptg))\simeq \Cbb G\mathrm{mod}$, where the codomain $\Cbb G\mathrm{mod}$ denotes the category of finite $\Cbb G$-modules. Thus for any $G$, we have $M(\ptg)\cong K_0(\Cbb G\mathrm{mod})$.
\item Let $\und{\al}\in\C(\xg,\yh)$ be any morphism.
\begin{itemize}
\item[{\rm (i)}] If $\und{\al}$ is an isomorphism, then $\el(\al)\uas$ and $\el(\al)_{!}$ are equivalences. Thus they induce homomorphisms
\[ M\uas(\und{\al})=K_0(\el(\al)\uas),\quad M_{!}(\und{\al})=K_0(\el(\al)_{!}) \]
on Grothendieck groups, which are mutually inverses.
\item[{\rm (ii)}] If $X=\pt$ and $Y=\pt$, then $\al=\frac{\ttt}{f}$ for some group homomorphism $f\co G\to H$.
Then the following diagram is commutative up to natural isomorphism.
\[
\xy
(-12,7)*+{\Dbb(\el(\ptg))}="0";
(12,7)*+{\Cbb G\mathrm{mod}}="2";
(-12,-7)*+{\Dbb(\el(\pth))}="4";
(12,-7)*+{\Cbb H\mathrm{mod}}="6";
{\ar^{\simeq} "0";"2"};
{\ar^{\el(\frac{\ttt}{f})\uas} "4";"0"};
{\ar_{\Hom_{\Cbb H}({}_{\Cbb H}\Cbb H_{\Cbb G},-)} "6";"2"};
{\ar_{\simeq} "4";"6"};
{\ar@{}|{\cong} "0";"6"};
\endxy
\]
Here ${}_{\Cbb H}\Cbb H_{\Cbb G}$ is a $\Cbb H$-$\Cbb G$-bimodule, whose right action is induced from $f$. The functor $\Hom_{\Cbb H}({}_{\Cbb H}\Cbb H_{\Cbb G},-)$ is exact. Moreover, since the characteristic of $\Cbb$ is zero, its left adjoint ${}_{\Cbb H}\Cbb H_{\Cbb G}\ot_{\Cbb G}-$ is also exact (\cite[1.1.4]{Bouc_biset}). Thus we obtain $M^{\ast}(\und{\al})=K_0(\el(\al)^{\ast})$ and $M_{!}(\und{\al})=K_0(\el(\al)_{!})$ also in this case.
\item[{\rm (iii)}] For a general case, we may take isomorphisms
\[ \zeta^{(\xg)}\co\xg\ov{\cong}{\lra}\frac{\pt}{G_1}\am\cdots\am\frac{\pt}{G_s},\ \ %
\zeta^{(\yh)}\co\yh\ov{\cong}{\lra}\frac{\pt}{H_1}\am\cdots\am\frac{\pt}{H_t}. \]
For any $1\le i\le s$, there exists unique $j_{(i)}$ such that $\al$ sends the component $\frac{\pt}{G_i}$ to $\frac{\pt}{H_{j_{(i)}}}$. Put $\al_i=\al|_{\frac{\pt}{G_i}}\co\frac{\pt}{G_i}\to\frac{\pt}{H_{j_{(i)}}}$, and put
\[ A\uas_{ij}=\begin{cases}\el(\al_i)\uas & (j=j_{(i)})\\ 0 & (j\ne j_{(i)})\end{cases} \]
by {\rm (ii)}. Then we define $M\uas(\al)$ to be the composition of
\begin{eqnarray*}
M(\xg)\!\!\!\!&\ov{M_{!}(\zeta^{(\xg)})}{\lra}&\!\!\!\! M(\frac{\pt}{G_1})\oplus\cdots \oplus M(\frac{\pt}{G_s})\\
&\ov{[A\uas_{ij}]_{ij}}{\lra}&\!\!\!\! M(\frac{\pt}{H_1})\oplus\cdots \oplus M(\frac{\pt}{H_t})\ov{M\uas(\zeta^{(\yh)})}{\lra}M(\yh).
\end{eqnarray*}
Similarly, we can define $M_{!}(\al)$ by using $\el(\al_i)_{!}$.
\end{itemize}
\end{enumerate}
\end{ex}

\section{Functors on the span category}

We introduce the {\it span category} $\Sp(\C)$ of $\C$, to which we can apply the result of Panchadcharam and Street (\cite{PS}).

\begin{dfn}\label{Def_0519_1}
Let $\xg,\yh\in\Ob(\C)$ be any pair of objects. A {\it span from $\xg$ to $\yh$} is a pair of morphisms in $\C$
\[ S=(\yh\ov{\und{\be}_S}{\lla}\frac{W_S}{L_S}\ov{\und{\al}_S}{\lra}\xg). \]
We abbreviate this to $S=(\spStu)$.

Let $S=(\spStu)$ and $T=(\und{\be}_T,\wlT,\und{\al}_T)$ be spans of morphisms in $\C$ from $\xg$ to $\yh$.
$S$ and $T$ are said to be {\it isomorphic} and written as $S\cong T$, if there is an isomorphism $\und{\gamma}\in\C(\wlS,\wlT)$ satisfying
\begin{equation}\label{Eqalbe}
\und{\al}_T\ci\und{\gamma}=\und{\al}_S \quad\text{and}\quad \und{\be}_T\ci\und{\gamma}=\und{\be}_S.
\end{equation}
We denote the set of isomorphism classes by $\Sp(\C)(\xg,\yh)$. The isomorphism class $[S]=[\spStu]$ of $S$ will be denoted by the corresponding lower case letter $s$.
%\[ [\und{S}]=[\spSu]=[\spStu]. \]
\end{dfn}

The class of natural weak pullbacks allows us to define the span category of $\C$.
\begin{dfn}\label{DefSpanCat}
Span category $\Scal=\Sp(\C)$ is defined as follows.
\begin{enumerate}
\item $\Ob(\Sp(\C))=\Ob(\C)$.
\item For any $\xg,\yh\in\Ob(\Sp(\C))$, the morphism set from $\xg$ to $\yh$ is $\Sp(\C)(\xg,\yh)$.
\end{enumerate}
When we want to emphasize $s$ is a morphism in $\Sp(\C)$, we will denote it as $s\co\xg\rta\yh$.
For a sequence of morphisms
\begin{eqnarray*}
&s=[\spStu]\co\xg\rta\yh&\\
&t=[\und{\be}_T,\wlT,\und{\al}_T]\co\yh\rta\zk,&
\end{eqnarray*}
their composition is defined to be $[\zk\ov{\und{\be}_T\ci\und{\wp}_T}{\lla}\wl\ov{\und{\al}_S\ci\und{\wp}_S}{\lra}\xg]$ by using a natural weak pullback as follows.
\[
\xy
(0,12)*+{\wl}="0";
(-15,0)*+{\frac{W_T}{L_T}}="2";
(15,0)*+{\frac{W_S}{L_S}}="4";
(-30,-12)*+{\zk}="6";
(0,-12)*+{\yh}="8";
(30,-12)*+{\xg}="10";
(0,0)*+{\nwp}="12";
{\ar_(0.6){\und{\wp}_T} "0";"2"};
{\ar^(0.6){\und{\wp}_S} "0";"4"};
{\ar_(0.6){\und{\be}_T} "2";"6"};
{\ar_(0.4){\und{\al}_T} "2";"8"};
{\ar^(0.4){\und{\be}_S} "4";"8"};
{\ar^(0.6){\und{\al}_S} "4";"10"};
\endxy
\]
The identity $\id\in\Sp(\C)(\xg,\xg)$ is given by $\id=[\xg\ov{\id}{\lla}\xg\ov{\id}{\lra}\xg]$.
\end{dfn}

\begin{rem}\label{RemTRemT}
Category $\Sp(\C)$ satisfies the following. (This agrees with $\Sp$ in \cite[Definition 5.2.10]{N_BisetMackey}.)
\begin{enumerate}
\item $\ems$ is zero object.
\item For any $\und{\al}\in\C(\xg,\yh)$, we can associate morphisms $\Rbf_{\und{\al}},\Tbf_{\und{\al}}$ in $\Sp(\C)$ by
\begin{eqnarray*}
&\Rbf_{\und{\al}}=[\xg\ov{\id}{\lla}\xg\ov{\und{\al}}{\lra}\yh]\in\Sp(\C)(\yh,\xg),&\\
&\Tbf_{\und{\al}}=[\yh\ov{\und{\al}}{\lla}\xg\ov{\id}{\lra}\xg]\in\Sp(\C)(\xg,\yh).&
\end{eqnarray*}
\item Coproduct $\am$ in $\C$ induces a product in $\Sp(\C)$. Namely, for any $\xg,\yh\in\Ob(\Sp(\C))$, if we take their coproduct
\[ \xg\ov{\und{\iota}_X}{\lra}\xg\am\yh\ov{\und{\iota}_Y}{\lla}\yh \]
in $\C$, then $\xg\ov{\Rbf_{\iota_X}}{\lta}\xg\am\yh\ov{\Rbf_{\iota_Y}}{\rta}\yh$ gives a product in $\Sp(\C)$.
\item $\Sp(\C)$ is enriched by $\Mon$. In fact, for any $\xg,\yh\in\Ob(\Sp(\C))$ and any $s=[\spSu], t=[\spTu]$, we can define their sum by
\[ s+t=[\yh\ov{\und{\be}_S\cup\und{\be}_T}{\lla}\wlS\am\wlT\ov{\und{\al}_S\cup\und{\al}_T}{\lra}\xg ], \]
which gives a structure of additive monoid on $\Sp(\C)(\xg,\yh)$ compatibly with compositions. The zero element in $\Sp(\C)(\xg,\yh)$ is given by $[\yh\leftarrow\ems\rightarrow\xg]$.
\end{enumerate}
\end{rem}

\bigskip

For categories $\Kcal,\Lcal$ with finite products, let $\Add(\Kcal,\Lcal)$ be the category of functors preserving finite products, where morphisms are natural transformations.

The following has been shown in \cite{N_BisetMackey}.
\begin{fact}\label{MackAdd}(\cite[Proposition 5.2.18]{N_BisetMackey})
There is an equivalence of categories
\[ \SMackC\simeq\Add(\Scal,\Sett). \]
\end{fact}
Details can be found in \cite{N_BisetMackey}. Since $\Scal$ is enriched by $\Mon$, any $F\in\Ob(\Add(\Scal,\Sett))$ factors through the forgetful functor $\Mon\to\Sett$. Thus we have equivalences
\[ \SMack(\C)\simeq\Add(\Scal,\Sett)\simeq\Add(\Scal,\Mon). \]
This equivalence induces an equivalence $\Mack^{\Zbb}(\C)\simeq\Add(\Scal,\Ab)$, and thus also
\begin{equation}\label{Equiv_Mack_Add}
\Mack^k(\C)\simeq\Add(\Scal,\kMod)
\end{equation}
for any commutative ring $k$.

We briefly state how the objects correspond by the equivalence $(\ref{Equiv_Mack_Add})$. 
For any $M=(M_{!},M^{\ast})\in\Ob(\Mack^k(\C))$, the corresponding object $F_M\in\Ob(\Add(\Scal,\kMod))$ is given by the following.
\begin{enumerate}
\item $F_M(\xg)=M(\xg)$ for any object $\xg\in\Ob(\Scal)$.
\item $F_M([\yh\ov{\und{\be}}{\lla}\wl\ov{\und{\al}}{\lra}\xg])=M_{!}(\be)\ci M^{\ast}(\al)$ for any morphism $[\yh\ov{\und{\be}}{\lla}\wl\ov{\und{\al}}{\lra}\xg]\in\Scal(\xg,\yh)$.
\end{enumerate}

\begin{dfn}\label{DefCatM}
Denote the category $\Add(\Scal,\kMod)$ by $\M$.

We call $F\in\Ob(\M)$ is {\it deflative} if it corresponds to a deflative Mackey functor by $(\ref{Equiv_Mack_Add})$. 
By the above correspondence, we see that $F\in\Ob(\M)$ is deflative if and only if it satisfies
\[ F(s)=\id \]
for any $\xg\in\Ob(\Scal)$ and any morphism $s=[\xg\ov{\und{\al}}{\lla}\wl\ov{\und{\al}}{\lra}\xg]\in\Scal(\xg,\xg)$ where $\al$ is stab-surjective.
We denote the full subcategory of deflative objects in $\M$ by $\M_d\se\M$.

Because of the equivalence in Fact \ref{MackAdd}, we mainly work on $\M$ instead of $\Mack^k(\C)$. We will denote $F_M\in\Ob(\M)$ abbreviately by the same symbol $M$ in the following.
\end{dfn}

\begin{ex}\label{ExCatM}
For $\Obigk\in\Ob(\MackCk)$ and $\Omk\in\Ob(\MackdCk)$, we abbreviately denote $F_{\Obigk}$ and $F_{\Omk}$ by $\Obigk\in\Ob(\M)$ and $\Omk\in\Ob(\M_d)$.

Similarly, we denote the morphism in $\M$ corresponding to $\pbf\in\MackCk(\Obigk,\Omk)$ (obtained in Proposition \ref{PropStrOrdBurn}) by the same symbol $\pbf\in\M(\Obigk,\Omk)$. This is an epimorphism in $\M$.
\end{ex}

\begin{rem}\label{RemClosedMd}
$\M_d\se\M$ is closed under isomorphisms and direct summands.
\end{rem}

\begin{prop}
The category $\Scal$ is a compact closed category. More precisely, it satisfies the following.
\begin{enumerate}
\item The functor $\iii\co\Scal^{\op}\to\Scal$ defined in the following way, is isomorphism of categories.
\begin{itemize}
\item[-] $\iii(\xg)=\xg$ for any object.
\item[-] $\iii([\yh\ov{\und{\be}}{\lla}\wl\ov{\und{\al}}{\lra}\xg])=[\xg\ov{\und{\al}}{\lla}\wl\ov{\und{\be}}{\lra}\yh]$ for any morphism.
\end{itemize}
Moreover, this isomorphism is involutive, i.e., it makes the following diagram commutative.
\[
\xy
(-12,0)*+{\Scal^{\op}}="0";
(0,8)*+{\Scal}="2";
(0,-6)*+{}="3";
(12,0)*+{\Scal^{\op}}="4";
{\ar^(0.46){\iii} "0";"2"};
{\ar^(0.54){\iii^{\op}} "2";"4"};
{\ar@/_0.8pc/_{\Id} "0";"4"};
{\ar@{}|\circlearrowright "2";"3"};
\endxy
\]
\item $\Scal$ is a compact closed category $($\cite[sections 5,6]{PS}$)$, with the following structure.
\begin{itemize}
\item[{\rm (i)}] Tensor product $-\ti-\co\Scal\ti\Scal\to\Scal$ is induced from the product in $\C$, in a natural way.
\item[{\rm (ii)}] Unit for the tensor product is given by $\pte$.
\item[{\rm (iii)}] Dual of $\xg\in\Ob(\Scal)$ is given by $\iii(\xg)=\xg$. Namely, for any $\xg,\yh,\zk\in\Ob(\Sp(\C))$, there is a natural bijection
\[ \Scal(\yh\ti\xg,\zk)\cong\Scal(\yh,\xg\ti\zk). \]
\end{itemize}
\end{enumerate}
\end{prop}
\begin{proof}
This can be shown in a similar way as in \cite[section 2]{PS}. We only have to replace pullbacks by natural weak pullbacks.
\end{proof}

Thus we can apply the results by Panchadcharam and Street \cite{PS}, to obtain the following.
\begin{cor}\label{PropPS1}$($\cite[sections 5,6]{PS}$)$
$\M$ is equipped with the following functors, which make $(\M,\ot,\Obigk)$ a symmetric closed monoidal category.
\begin{enumerate}
\item Tensor functor $-\ot-\co\M\ti\M\to\M$.
\item Internal Hom functor $\HH\co(\M)^{\op}\ti\M\to\M$, together with a natural isomorphism
\[ \M(M\ot N,P)\cong\M(M,\HH(N,P))\quad(\fa M,N,P\in\Ob(\M)). \]
\end{enumerate}

We use the following symbols for the natural isomorphisms which are part of the monoidal structure $($\cite[VII.1]{MacLane}$)$. For any $L,M,N\in\Ob(\M)$,
\begin{equation}\label{Name_as}
\as\co L\ot (M\ot N)\ov{\cong}{\lra}(L\ot M)\ot N,
\end{equation}
\begin{equation}\label{Name_sym}
\sym\co M\ot N\ov{\cong}{\lra}N\ot M
\end{equation}
$($we abbreviately use the same symbol independently from $L,M,N$ as above$)$, and
\begin{equation}\label{Name_lr}
\ell_M\co \Obigk\ot M\ov{\cong}{\lra}M,\quad r_M\co M\ot \Obigk\ov{\cong}{\lra}M.
\end{equation}
\end{cor}

\bigskip

\bigskip

In the proceeding sections, we will show the following {\rm (I),(II),(III)}.
\begin{enumerate}
\item[{\rm (I)}] Functors $-\ot-$ and $\HH$ restricts to yield
\[ -\otd-\co\M_d\ti\M_d\to\M_d\quad\text{and}\quad %
\HH_d\co(\M_d)^{\op}\ti\M_d\to\M_d,
\]
which make $(\M_d,\otd,\Omk)$ a symmetric closed monoidal category.
\item[{\rm (II)}] $\Omk$ is a monoid in $\M$. Moreover, the category $\OMod$ of $\Omk$-modules is equivalent to $\M_d$.
\item[{\rm (III)}] The equivalence $\Phi\co\M_d\ov{\simeq}{\lra}\BFk$ given in \cite{N_BisetMackey} is a monoidal equivalence.
\end{enumerate}
In the sequel, we sometimes refer to these {\rm (I),(II),(III)}.
To prove them, let us review the construction used in \cite{PS}, in our terminology.
\begin{dfn}\label{DefDressbyX}
$\ \ $
\begin{enumerate}
\item Let $\xg\in\Ob(\Scal)$ be any object. Then the functor
\[ \ppp_{\xg}=-\ti\xg\co\Scal\to\Scal \]
preserves finite products (given by coproducts in $\C$) in $\Scal$, and thus gives a functor
\[ \Dbf(\xg,-)=-\ci\ppp_{\xg}\co\M\to\M. \]
For an object $M\in\Ob(\M)$, we simply denote $\Dbf(\xg,M)=M\ci\ppp_{\xg}\in\Ob(\M)$ by $M_{\xg}$.
\item Let $s\in\Scal(\xg,\yh)$ be any morphism. This induces a natural transformation
\[ \ppp_s=-\ti s\co\ppp_{\xg}\tc\ppp_{\yh}, \]
and thus it gives a natural transformation
\[ \Dbf(s,-)\co\Dbf(\xg,-)\tc\Dbf(\yh,-). \]
For any object $M\in\Ob(\M)$, we denote $\Dbf(s,M)\co\Dbf(\xg,M)\to\Dbf(\yh,M)$ simply by $M_s\co M_{\xg}\to M_{\yh}$.
\end{enumerate}
\end{dfn}
\begin{rem}
In fact, we have a functor $\Dbf\co\Scal\ti\M\to \M$.
\end{rem}

\begin{dfn}\label{DefTensor}(\cite[section 5]{PS})
Let $M,N\in\Ob(\M)$ be any pair of objects. Then $M\ot N\in\Ob(\M)$ is defined in the following way.
\begin{enumerate}
\item For any $\xg\in\Ob(\Scal)$, define a functor $T^{(\xg)}\co(\M)^{\op}\ti\M\to\M$ to be the composition of
\[ (\Scal)^{\op}\ti\Scal\ov{\iii\ti\Id}{\lra}\Scal\ti\Scal\ov{M_{\xg}\ti N}{\lra}\kMod\ti\kMod\ov{-\underset{k}{\ot}-}{\lra}\kMod. \]
Define $(M\ot N)(\xg)=\int^{\yh}T^{(\xg)}(\yh,\yh)$ to be its coend, equipped with the universal dinatural transformation
\[ \om^{(\xg)}\co T^{(\xg)}\ov{..}{\lra}(M\ot N)(\xg). \]
For the notation related to coends, see \cite[IX.6]{MacLane}.
\item For any morphism $s\in\Scal(\xg,\yh)$, the morphism $M_s\in\M(M_{\xg},M_{\yh})$ gives a natural transformation defined by the horizontal composition (= \lq whiskering' (\cite[P.275]{MacLane})) of the following.
\begin{equation}\label{EditAdd2}
\xy
(-46,0)*+{(\Scal)^{\op}\ti\Scal}="0";
(-20,0)*+{\Scal\ti\Scal}="2";
(20,0)*+{\kMod\ti\kMod}="4";
(46,0)*+{\kMod}="6";
{\ar^(0.56){\iii\ti\Id} "0";"2"};
{\ar@/^1.2pc/^{M_{\xg}\ti N} "2";"4"};
{\ar@/_1.2pc/_{M_{\yh}\ti N} "2";"4"};
{\ar^(0.6){-\underset{k}{\ot}-} "4";"6"};
{\ar@{=>}_{M_s\ti\id_N} (2,2);(2,-2)};
\endxy
\end{equation}
The universal property of the coend induces a morphism
\[ (M\ot N)(s)\co(M\ot N)(\xg)\to(M\ot N)(\yh) \]
compatible with $\om^{(\xg)}$ and $\om^{(\yh)}$.
\end{enumerate}
This makes $M\ot N$ into a functor $M\ot N\co\Scal\to\kMod$. This preserves finite products, and thus we obtain $M\ot N\in\Ob(\M)$.

For any $\varphi\in\M(M,M\ppr)$ and $\psi\in\M(N,N\ppr)$, the morphism $\varphi\ot\psi\co M\ot N\to M\ppr\ot N\ppr$ is also induced by the universal property of the coend.
\end{dfn}

\begin{rem}\label{DefTensorObjValue}
Explicitly, $(M\ot N)(\xg)$ is given by the following way, in terms of $\C$.
Let $M,N$ be objects in $\M$. For any $\xg\in\Ob(\C)$, the $k$-module $(M\ot N)(\xg)$ is given by 
\[ (M\ot N)(\xg)=(\underset{\akaxgu}{\bigoplus}M(\ak)\ot_k N(\ak))/\Isc, \]
where
\begin{itemize}
%\item[-] $\ot_k$ denotes the tensor product of $k$-modules over $k$.
\item[-] the direct sum runs over objects $(\akaxgu)$ in $\C/\xg$,
%1-1-2
\item[-] $\Isc$ is a $k$-submodule generated by
\begin{eqnarray*}
&\Set{ M\uas(\vp)(m\ppr)\ot n-m\ppr\ot N_{!}(\vp)(n)\ |\ \begin{array}{c}\und{\vp}\in (\C/\xg)(\akaxgu,\akaxgpu),\\ m\ppr\in M(\akp),n\in N(\ak)\end{array}}\\
\cup &\Set{M_{!}(\vp)(m)\ot n\ppr-m\ot N\uas(\vp)(n\ppr)\ |\ \begin{array}{c}\und{\vp}\in (\C/\xg)(\akaxgu, \akaxgpu),\\ m\in M(\ak), n\ppr \in N(\akp)\end{array}}.
\end{eqnarray*}
\end{itemize}
Here, $\C/\xg$ denotes the slice category of $\C$ over $\xg$.
\end{rem}

\begin{rem}\label{DefUnit}
The unit for the tensor is given by the composition of
\[ \Scal\ov{\Scal(\pte,-)}{\lra}\Mon\ov{K_0}{\lra}\Ab\ov{k\underset{\Zbb}{\ot}-}{\lra}\kMod, \]
which agrees with $\Obigk$.
\end{rem}

\begin{dfn}\label{DefHH}(\cite[section 6]{PS})
Let $M,N\in\Ob(\M)$ be any pair of objects. Then $\HH(M,N)\in\Ob(\M)$ is defined in the following way.
\begin{enumerate}
\item For any $\xg\in\Ob(\Scal)$, define as
\[ (\HH(M,N))(\xg)=\M(M,N_{\xg}). \]
\item For any $s\in\Scal(\xg,\yh)$, the morphism $N_s\co N_{\xg}\to N_{\yh}$ induces a morphism
\[ (\HH(M,N))(s)=N_s\ci-\co\M(M,N_{\xg})\to\M(M,N_{\yh}). \]
\end{enumerate}
This gives an object $\HH(M,N)\in\Ob(\M)$.

For any $\varphi\in\M(M,M\ppr)$ and $\psi\in\M(N,N\ppr)$, the morphism $\HH(\varphi,\psi)\co\HH(M\ppr,N)\to\HH(M,N\ppr)$ is defined by the composition of
\[ \M(M\ppr,N_{\xg})\ov{-\ci\varphi}{\lra}\M(M,N_{\xg})\ov{\Dbf(\xg,\psi)\ci-}{\lra}\M(M,N_{\xg}\ppr) \]
at any $\xg\in\Ob(\Scal)$.
\end{dfn}

\section{Monoidal structure on $\M_d$}

In this section, we show {\rm (I)}.

\begin{lem}\label{LemDressDefl}
For any $M\in\Ob(\M_d)$, we have the following.
\begin{enumerate}
\item Let $s=[\xg\ov{\und{\al}}{\lla}\wl\ov{\und{\al}}{\lra}\xg]$ be any morphism in $\Scal$, where $\al$ is stab-surjective. Then we have $M_s=\id\co M_{\xg}\to M_{\xg}$.
\item For any $\xg\in\Ob(\Scal)$, we have $M_{\xg}\in\Ob(\M_d)$.
\end{enumerate}
\end{lem}
\begin{proof}
This is an immediate consequence of Remark \ref{RemStabsurj} {\rm (3)}, because products are (special case of) natural weak pullbacks.
\end{proof}

\begin{prop}\label{PropTensorRestrict}
Let $M,N\in\Ob(\M)$ be any pair of objects. If at least one of $M$ and $N$ belongs to $\M_d$, then we have $M\ot N\in\Ob(\M_d)$.
\end{prop}
\begin{proof}
By the symmetry, we may assume $M\in\Ob(\M_d)$. Let $s=[\xg\ov{\und{\al}}{\lla}\wl\ov{\und{\al}}{\lra}\xg]$ be any morphism in $\Scal$, where $\al$ is stab-surjective. By Lemma \ref{LemDressDefl} {\rm (1)}, we have $M_s=\id$. Thus the corresponding natural transformation $(\ref{EditAdd2})$ in Definition \ref{DefTensor} becomes identity. By the universal property of the coend, this implies $(M\ot N)(s)=\id$.
\end{proof}

\begin{dfn}\label{DefOtd}
By Proposition \ref{PropTensorRestrict}, the tensor functor on $\M$ restricts to yield a functor
\[ -\otd-\co\M_d\ti\M_d\to\M_d. \]
\end{dfn}

\begin{prop}\label{PropHomDefl}
Let $M,N\in\Ob(\M)$ be any pair of objects. If $N\in\Ob(\M_d)$, then $\HH(M,N)\in\Ob(\M_d)$.
\end{prop}
\begin{proof}
Let $s=[\xg\ov{\und{\al}}{\lla}\wl\ov{\und{\al}}{\lra}\xg]$ be any morphism in $\Scal$, where $\al$ is stab-surjective. By Lemma \ref{LemDressDefl} {\rm (1)}, we have $N_s=\id$. This implies $(\HH(M,N))(s)=\id$.
\end{proof}

\begin{dfn}\label{DefHd}
By Proposition \ref{PropHomDefl}, the functor $\HH$ restricts to yield a functor
\[ \HH_d\co(\M_d)^{\op}\ti\M_d\to\M_d. \]
\end{dfn}

\begin{cor}\label{CorHomDefl}
For any $M,N,P\in\Ob(\M_d)$, we have a natural isomorphism
\[ \M_d(M\otd N,P)\cong\M_d(M,\HH_d(N,P)). \]
In particular, the endofunctor $-\otd N\co\M_d\to\M_d$ is left adjoint to $\HH_d(N,-)$, for any $N\in\Ob(\M_d)$.
\end{cor}
\begin{proof}
This immediately follows from Propositions \ref{PropHomDefl}, \ref{DefHd} and Corollary \ref{PropPS1}.
\end{proof}

\begin{rem}\label{RemEpiTensor}
Let $\pbf\in\M(\Obigk,\Omk)$ be the epimorphism in Example \ref{ExCatM}. By the closedness, for any $M\in\Ob(\M)$, it follows that $\pbf\ot M$ is also an epimorphism. By the symmetry, so is $M\ot \pbf$. In particular $\pbf\ot\pbf\co\Obigk\ot\Obigk\to\Omk\ot\Omk$ is epimorphism, since $\pbf\ot\pbf=(\pbf\ot\Omk)\ci(\Obigk\ot\pbf)$.
\end{rem}

\begin{prop}\label{LemOmegapDefl}
For any $M\in\Ob(\M_d)$, the epimorphism $\pbf\co\Obigk\to\Omk$ induces an isomorphism
\begin{equation}
-\ci\pbf\co\M(\Omk,M)\ov{\cong}{\lra}\M(\Obigk,M).
\end{equation}
\end{prop}
\begin{proof}
This follows from Proposition \ref{PropHHom}, via the equivalence $\M\simeq\MackCk$.
\end{proof}

\begin{prop}\label{PropEquivDefl}
For any $M\in\Ob(\M_d)$,
\[ \HH(\pbf,M)\co\HH(\Omk,M)\to\HH(\Obigk,M) \]
is isomorphism in $\M$.
\end{prop}
\begin{proof}
By definition, for each $\xg\in\Ob(\Scal)$,
\[ (\HH(\pbf,M))_{\xg}\co(\HH(\Omk,M))(\xg)\to(\HH(\Obigk,M))(\xg) \]
is given by $\M(\pbf,M_{\xg})\co\M(\Omk,M_{\xg})\to\M(\Obigk,M_{\xg})$.
This is isomorphism, by Lemma \ref{LemDressDefl} and Proposition \ref{LemOmegapDefl}.
\end{proof}

\begin{cor}\label{CorEquivDefl1}
Let $M\in\Ob(\M)$ be any object. Let $\ell_M\co\Obigk\ot M\ov{\cong}{\lra}M$ be the isomorphism $(\ref{Name_lr})$. Let $v_M\co\Obigk\to\HH(M,M)$ be the morphism corresponding to $\ell_M$ by the adjoint property. Then the following are equivalent.
\begin{enumerate}
\item $M\in\Ob(\M_d)$.
\item $\ell_M$ factors through $\pbf\ot M$. Namely, there exists $\ell_M\ppr\co\Omk\ot M\to M$ satisfying $\ell_M\ppr\ci(\pbf\ot M)=\ell_M$.
\item $v_M$ factors through $\pbf$. Namely, there exists $v_M\ppr\co\Omk\to \HH(M,M)$ satisfying $v_M\ppr\ci\pbf=v_M$.
\end{enumerate}
Remark that $v\ppr_M$ and $\ell_M\ppr$ in {\rm (2),(3)} are unique if they exist, since $\pbf$ and $\pbf\ot M$ are epimorphisms.
\end{cor}
\begin{proof}
$(1)\Rightarrow(3)$ follows from Propositions \ref{PropHomDefl} and \ref{LemOmegapDefl}. $(2)\LR(3)$ follows from the adjoint property.

If $(2)$ holds, then $M$ becomes a direct summand of $\Omk\ot M$, which belongs to $\M_d$ by Proposition \ref{PropTensorRestrict}. Thus $M$ satisfies $M\in\Ob(\M_d)$ by Remark \ref{RemClosedMd}.
\end{proof}

For any $M\in\Ob(\M_d)$, the morphism $\ell_M\ppr$ obtained in Corollary \ref{CorEquivDefl1} {\rm (2)} becomes isomorphism. More precisely, we have the following.
\begin{cor}\label{CorEquivDefl2}
For any $M\in\Ob(\M)$, the following are equivalent.
\begin{enumerate}
\item $M\in\Ob(\M_d)$.
\item There is an isomorphism $\ell_M\ppr\co\Omk\ot M\ov{\cong}{\lra} M$ satisfying $\ell_M\ppr\ci(\pbf\ot M)=\ell_M$.
\item There exists an isomorphism $\Omk\ot M\cong M$.
\end{enumerate}
\end{cor}
\begin{proof}
$(2)\Rightarrow(3)$ is trivial. $(3)\Rightarrow(1)$ follows from Remark \ref{RemClosedMd} and Proposition \ref{PropTensorRestrict}.

Let us show $(1)\Rightarrow(2)$.
Suppose $M$ belongs to $\M_d$. By Corollary \ref{CorEquivDefl1}, we have a morphism $\ell_M\ppr$ satisfying $\ell_M\ppr\ci(\pbf\ot M)=\ell_M$. Let us show $\ell_M\ppr$ is isomorphism.
Remark that $\Omk\ot M=\Omk\otd M$ belongs to $\M_d$ by Proposition \ref{PropTensorRestrict}. For any $N\in\Ob(\M)$, the following diagram is commutative.
\[
\xy
(-42,8)*+{\M(\Obigk,\HH(M,N))}="0";
(0,8)*+{\M(\Obigk\ot M,N)}="2";
(-42,-8)*+{\M(\Omk,\HH(M,N))}="10";
(0,-8)*+{\M(\Omk\otd M,N)}="12";
(12,6)*+{}="13";
(34,-8)*+{\M(M,N)}="14";
{\ar^{\cong} "0";"2"};
{\ar^{\M(\pbf,\HH(M,N))} "10";"0"};
{\ar^{\M(\pbf\ot M,N)} "12";"2"};
{\ar_{\M(\ell_M,N)}^{\cong} "14";"2"};
{\ar_{\cong} "10";"12"};
{\ar^(0.42){\M(\ell_M\ppr,N)} "14";"12"};
{\ar@{}|\circlearrowright "0";"12"};
{\ar@{}|\circlearrowright "12";"13"};
\endxy
\]
If $N\in\Ob(\M_d)$, then $\M(\pbf,\HH(M,N))$ is isomorphism by Propositions \ref{PropHomDefl}, \ref{LemOmegapDefl}. By the above commutativity, it follows that $\M(\ell_M\ppr,N)=\M_d(\ell_M\ppr,N)$ is an isomorphism for any $N\in\Ob(\M_d)$. Thus $\ell_M\ppr\in\M_d(\Omk\otd M,M)$ is isomorphism by Yoneda's lemma.
\end{proof}

\begin{rem}\label{RemEquivDefl}
Since $\HH(\Obigk,-)\co\M\to\M$ is isomorphic to the identity functor, Proposition \ref{PropEquivDefl} implies the following. ({\rm (ii)} also follows from Corollary \ref{CorEquivDefl2}.)

\begin{itemize}
\item[{\rm (i)}] $\HH(\Omk,-)|_{\M_d}\co\M_d\to\M$ is isomorphic to the inclusion functor $\M_d\hookrightarrow\M$.
\item[{\rm (ii)}] $\HH_d(\Omk,-)\co\M_d\to\M_d$ is isomorphic to the identity functor $\Id_{\M_d}$.
\end{itemize}
\end{rem}

\begin{cor}\label{CorDefCor}
Tensoring with $\Omk$ gives a functor
\[ -\ot\Omk\co\M\to\M_d, \]
which is left adjoint to the inclusion $\M_d\hookrightarrow\M$.
\end{cor}
\begin{proof}
This follows from Proposition \ref{PropTensorRestrict}, Corollary \ref{CorHomDefl} and Remark \ref{RemEquivDefl}.
\end{proof}

By Corollary \ref{CorEquivDefl2}, we have isomorphism $\ell_M\ppr\co\Omk\ot M\ov{\cong}{\lra}M$ for any $M\in\Ob(\M_d)$. The commutativity $\ell_M\ppr\ci(\pbf\ot M)=\ell_M$ and the naturality of $\ell$ ensures the naturality of $\ell\ppr$ in $M$. (This is shown in a similar way as for the commutativity of $(\ref{TTTTT})$ below. See the proof of Proposition \ref{PropDeflMon}.)
\begin{dfn}\label{Defr}
For any $M\in\Ob(\M_d)$, define the isomorphism $r_M\ppr\co M\ot\Omk\ov{\cong}{\lra}M$ to be the composition of isomorphisms
\[ M\ot\Omk\ov{\sym}{\lra}\Omk\ot M\ov{\ell_M\ppr}{\lra}M. \]
This is natural in $M$, and satisfies $r_M\ppr\ci(M\ot\pbf)=r_M$ where $r_M\co M\ot\Obigk\ov{\cong}{\lra}M$ is the isomorphism $(\ref{Name_lr})$.
\end{dfn}

\begin{prop} \label{PropDeflMon}
$(\M_d,\otd,\Omk)$ is a symmetric closed monoidal category. In addition, the functor $-\ot\Omk \co \M\to\M_d$ is a monoidal functor.
\end{prop}
\begin{proof}
First, we show that $(\M_d,\otd,\Omk)$ is a symmetric monoidal category. For any $L,M,N\in\Ob(\M_d)$, the isomorphisms
\begin{eqnarray*}
&\as\co L\otd(M\otd N)\ov{\cong}{\lra}(L\otd M)\otd N,&\\
&\sym\co M\otd N\ov{\cong}{\lra}N\otd M&
\end{eqnarray*}
are taken to be the same as those for $\M$ (i.e. $(\ref{Name_as}),(\ref{Name_sym})$). Thus the compatibility among them is inherited from that in $\M$.
It remains to show the commutativity of 
\begin{equation}\label{TTTTT}
\xy
(-15,8)*+{M\otd(\Omk\otd N)}="0";
(15,8)*+{(M\otd\Omk)\otd N}="2";
(0,-8)*+{M\otd N}="4";
(0,11)*+{}="5";
{\ar^{\as}_{\cong} "0";"2"};
{\ar_(0.4){M\otd \ell_N\ppr} "0";"4"};
{\ar^(0.4){r_M\ppr\otd N} "2";"4"};
{\ar@{}|\circlearrowright "4";"5"};
\endxy.
\end{equation}
This follows from the commutativity of the following diagram and the epimorphicity of $M\ot(\pbf\ot N)$.
\[
\xy
(-15,8)*+{M\ot(\Obigk\ot N)}="0";
(-44,20)*+{M\otd(\Omk\otd N)}="10";
(15,8)*+{(M\ot\Obigk)\ot N}="2";
(44,20)*+{(M\otd\Omk)\otd N}="12";
(0,-8)*+{M\ot N}="4";
(0,-24)*+{M\otd N}="14";
(0,11)*+{}="5";
{\ar^{\as} "0";"2"};
{\ar^{\as} "10";"12"};
{\ar_(0.4){M\ot \ell_N} "0";"4"};
{\ar_(0.46){M\otd \ell_N\ppr} "10";"14"};
{\ar^(0.4){r_M\ot N} "2";"4"};
{\ar^(0.46){r_M\ppr\otd N} "12";"14"};
{\ar_{M\ot(\pbf\ot N)} "0";"10"};
{\ar^{(M\ot\pbf)\ot N} "2";"12"};
{\ar@{=} "4";"14"};
(-14,-6)*+{_{\circlearrowright}}="-1";
(14,-6)*+{_{\circlearrowright}}="-2";
(0,15)*+{_{\circlearrowright}}="-3";
{\ar@{}|\circlearrowright "4";"5"};
\endxy
\]
Closedness follows from Corollary \ref{CorHomDefl}.

For any $M,N\in\Ob(\M)$, we have isomorphisms
\[ (M\ot\Omk)\otd(N\ot\Omk)\cong (M\ot N)\ot(\Omk\otd\Omk)\ov{\id\ot \ell\ppr_{\Omk}}{\lra}(M\ot N)\ot\Omk \]
and
\[ \ell_{\Omk}\co\Obigk\ot\Omk\ov{\cong}{\lra}\Omk. \]
With these isomorphisms, we can confirm $-\ot\Omk\co\M\to\M_d$ is a monoidal functor, in a straightforward way.
\end{proof}

\section{Equivalence $\M_d\simeq\OMod$}

In this section, we show {\rm (II)}. Although this is a formal consequence of the epimorphicity of $\pbf$ and Corollary \ref{CorEquivDefl2}, we give a short proof for the sake of completeness.

As in the case of ordinary Mackey functors, we call a monoid object in $\M$ a Green functor.
\begin{dfn}\label{DefGreenS}
A {\it Green functor} on $\C$ is an object $\Gam\in\Ob(\M)$ equipped with morphisms
\[ m=m_{\Gam}\co\Gam\ot\Gam\to\Gam,\quad u=u_{\Gam}\co\Obigk\to\Gam \]
which make the following diagrams commutative.
\begin{itemize}
\item[{\rm (i)}]
$
\xy
(-32,8)*+{\Gam\ot(\Gam\ot\Gam)}="0";
(0,8)*+{(\Gam\ot\Gam)\ot\Gam}="2";
(32,8)*+{\Gam\ot\Gam}="4";
(-32,-8)*+{\Gam\ot\Gam}="10";
(32,-8)*+{\Gam}="14";
{\ar^{\as}_{\cong} "0";"2"};
{\ar^(0.56){m\ot\Gam} "2";"4"};
{\ar_{\Gam\ot m} "0";"10"};
{\ar^{m} "4";"14"};
{\ar_{m} "10";"14"};
{\ar@{}|\circlearrowright "0";"14"};
\endxy
$
\item[{\rm (ii)}]
$ 
\xy
(-24,8)*+{\Obigk\ot\Gam}="-2";
(-11,-5)*+{}="-1";
(0,8)*+{\Gam\ot\Gam}="0";
(11,-5)*+{}="1";
(24,8)*+{\Gam\ot\Obigk}="2";
(0,-8)*+{\Gam}="4";
{\ar^{u\ot\Gam} "-2";"0"};
{\ar_{\Gam\ot u} "2";"0"};
{\ar_{\ell_{\Gam}}^{\cong} "-2";"4"};
{\ar^{r_{\Gam}}_{\cong} "2";"4"};
{\ar|*+{_{m}} "0";"4"};
{\ar@{}|\circlearrowright "0";"-1"};
{\ar@{}|\circlearrowright "0";"1"};
\endxy
$
\end{itemize}
A Green functor $\Gam=(\Gam,m,u)$ is {\it commutative}, if it satisfies $m\ci\sym=m$.

If $\Gam=(\Gam,m_{\Gam},u_{\Gam})$ and $\Lam=(\Lam,m_{\Lam},u_{\Lam})$ are two Green functors on $\C$, then a {\it morphism} $f\co\Gam\to\Lam$ of Green functors is a morphism $f\in\M(\Gam,\Lam)$ which satisfies $f\ci m_{\Gam}=m_{\Lam}\ci(f\ot f)$ and $f\ci u_{\Gam}=u_{\Lam}$.
The composition and the identities are naturally induced from those in $\M$. The category of Green functors on $\C$ is denoted by $\GreenCk$.
\end{dfn}

\begin{dfn}
A Green functor $\Gam$ is called {\it deflative}, if it is deflative as an object in $\M$. We denote the full subcategory of $\GreenCk$ consisting of deflative ones by $\GreendCk$. Since $\Obigk$ is the unit of the tensor, it gives a Green functor $(\Obigk,\ell_{\Obigk},\id)$, which is initial in $\GreenCk$. %Here, $\ell$ denotes the isomorphism $\ell\co\Obigk\ot\Obigk\ov{\cong}{\lra}\Obigk$.
\end{dfn}

\begin{dfn}\label{DefModule}
Let $\Gam=(\Gam,m,u)$ be a Green functor on $\C$.
A $\Gam${\it -module} is an object $M\in\Ob(\M)$ equipped with a morphism $\ac=\ac_{M}\co\Gam\ot M\to M$ which makes the following diagrams commutative.
\begin{itemize}
\item[{\rm (i)}]
$
\xy
(-32,8)*+{\Gam\ot(\Gam\ot M)}="0";
(0,8)*+{(\Gam\ot\Gam)\ot M}="2";
(32,8)*+{\Gam\ot M}="4";
(-32,-8)*+{\Gam\ot M}="10";
(32,-8)*+{M}="14";
{\ar^{\as} "0";"2"};
{\ar^(0.54){m\ot M} "2";"4"};
{\ar_{\Gam\ot\ac} "0";"10"};
{\ar^{\ac} "4";"14"};
{\ar_{\ac} "10";"14"};
{\ar@{}|\circlearrowright "0";"14"};
\endxy
$
\item[{\rm (ii)}]
$\xy
(-24,8)*+{\Obigk\ot M}="-2";
(-11,-5)*+{}="-1";
(0,8)*+{\Gam\ot M}="0";
(0,-8)*+{M}="4";
{\ar^{u\ot M} "-2";"0"};
{\ar_{\ell_M} "-2";"4"};
{\ar^{\ac} "0";"4"};
{\ar@{}|\circlearrowright "0";"-1"};
\endxy
$
\end{itemize}
If $M$ and $N$ are $\Gam$-modules, then a {\it morphism} $f\co M\to N$ of $\Gam$-modules is $f\in\M(M,N)$ which satisfies $f\ci\ac_M=\ac_N\ci(\Gam\ot f)$.
The compositions and the identities are naturally induced from those in $\M$. We denote the category of $\Gam$-modules by $\Gam\Mod$.
\end{dfn}

\begin{lem}\label{CommutForMd}
For any $M\in\Ob(\M_d)$, the following diagram is commutative.
\begin{equation}\label{CommForM}
\xy
(-32,8)*+{\Omk\otd(\Omk\otd M)}="0";
(0,8)*+{(\Omk\otd\Omk)\otd M}="2";
(32,8)*+{\Omk\otd M}="4";
(-32,-8)*+{\Omk\otd M}="10";
(32,-8)*+{M}="14";
{\ar^{\as} "0";"2"};
{\ar^(0.58){\ell_{\Omk}\ppr\otd M} "2";"4"};
{\ar_{\Omk\otd\ell_M\ppr} "0";"10"};
{\ar^{\ell_M\ppr} "4";"14"};
{\ar_{\ell_M\ppr} "10";"14"};
{\ar@{}|\circlearrowright "0";"14"};
\endxy
\end{equation}
\end{lem}
\begin{proof}
Remark that
\[
\xy
(-20,8)*+{\Omk\otd(\Omk\otd M)}="0";
(20,8)*+{\Obigk\ot(\Obigk\ot M)}="2";
(-20,-8)*+{\Omk\otd M}="4";
(20,-8)*+{\Obigk\ot M}="6";
{\ar_(0.52){\pbf\ot(\pbf\ot M)} "2";"0"};
{\ar_{\Omk\otd\ell_M\ppr} "0";"4"};
{\ar^{\Obigk\ot\ell_M} "2";"6"};
{\ar^{\pbf\ot M} "6";"4"};
{\ar@{}|\circlearrowright "0";"6"};
\endxy
\]
is commutative by the equality $\ell_M\ppr\ci(\pbf\ot M)=\ell_M$ and the functoriality of the tensor product. Similarly,
\[
\xy
(-18,8)*+{\Obigk\ot(\Obigk\ot M)}="0";
(18,8)*+{(\Obigk\ot\Obigk)\ot M}="2";
(-18,-8)*+{\Omk\otd(\Omk\ot M)}="4";
(18,-8)*+{(\Omk\otd\Omk)\ot M}="6";
{\ar^{\as} "0";"2"};
{\ar_{\pbf\ot(\pbf\ot M)} "0";"4"};
{\ar^{(\pbf\ot\pbf)\ot M} "2";"6"};
{\ar_{\as} "4";"6"};
{\ar@{}|\circlearrowright "0";"6"};
\endxy
\ \ \text{and}\ \ 
\xy
(-16,8)*+{(\Obigk\ot\Obigk)\ot M}="0";
(16,8)*+{\Obigk\ot M}="2";
(-16,-8)*+{(\Omk\otd\Omk)\otd M}="4";
(16,-8)*+{\Omk\otd M}="6";
{\ar^(0.6){\ell_{\Obigk}\ot M} "0";"2"};
{\ar_{(\pbf\ot\pbf)\ot M} "0";"4"};
{\ar^{\pbf\ot M} "2";"6"};
{\ar_(0.6){\ell_{\Omk}\ppr\otd M} "4";"6"};
{\ar@{}|\circlearrowright "0";"6"};
\endxy
\]
are commutative. Now the commutativity of $(\ref{CommForM})$ follows from the commutativity of
\[
\xy
(-36,8)*+{\Obigk\ot(\Obigk\ot M)}="0";
(0,8)*+{(\Obigk\ot\Obigk)\ot M}="2";
(36,8)*+{\Obigk\ot M}="4";
(-36,-8)*+{\Obigk\ot M}="10";
(36,-8)*+{M}="14";
{\ar^{\as} "0";"2"};
{\ar^(0.6){\ell_{\Obigk}\ot M} "2";"4"};
{\ar_{\Obigk\ot\ell_M} "0";"10"};
{\ar^{\ell_M} "4";"14"};
{\ar_{\ell_M} "10";"14"};
{\ar@{}|\circlearrowright "0";"14"};
\endxy
\]
and the epimorphicity of $\pbf\ot(\pbf\ot M)$, by a similar argument as for the commutativity of $(\ref{TTTTT}$).
\end{proof}

\begin{prop}\label{PropBurnside}
$(\Omk, \ell_{\Omk}\ppr, \pbf)$ is a deflative commutative Green functor on $\C$.
Moreover, $\pbf\co\Obigk\to\Omk$ is a morphism of Green functors.
\end{prop}
\begin{proof}
This follows from the definition of $\ell\ppr$ and the commutativity of $(\ref{CommForM})$ applied to $M=\Omk$.
\end{proof}

\begin{prop} \label{PropActOmega}
For any $M\in\Ob(\M)$, the following are equivalent.
\begin{enumerate}
\item $M\in\Ob(\M_d)$.
\item $M$ has a $($unique$)$ structure of an $\Omk$-module.
\end{enumerate}
\end{prop}
\begin{proof}
By Corollary \ref{CorEquivDefl1}, $M$ belongs to $\M_d$ if and only if there is a unique morphism $\ell\ppr_M\co\Omk\ot M\to M$ satisfying $\ell\ppr_M\ci(\pbf\ot M)=\ell_M$. This commutativity corresponds to the commutativity of {\rm (ii)} in Definition \ref{DefModule}, for $(\Gam,m,u)=(\Omk,\ell_{\Omk}\ppr,\pbf)$.
Moreover, if $M\in\Ob(\M_d)$, the commutativity of $(\ref{CommForM})$ means the commutativity of {\rm (i)} in Definition \ref{DefModule}.
Uniqueness of the $\Omk$-module structure on $M$ follows from the uniqueness of $\ell\ppr_M$ in Corollary \ref{CorEquivDefl1}.
\end{proof}

\begin{cor} \label{ThmEquivOModDefl}
There is an equivalence of categories $\M_d\simeq\OMod$.
\end{cor}
\begin{proof}
By Proposition \ref{PropActOmega}, each of these categories can be viewed as a subcategory of $\M$, whose class of objects is equal to $\Ob(\M_d)$. What remains to show is that the inclusion
\[ \OMod\hookrightarrow\M_d \]
is full. Namely, it suffices to show the following.
\begin{itemize}
\item[-] For any $M,N\in\Ob(\M_d)$, any morphism $f\in\M_d(M,N)$ makes the following diagram commutative.
\[
\xy
(-12,7)*+{\Omk\ot M}="0";
(12,7)*+{\Omk \ot N}="2";
(-12,-7)*+{M}="4";
(12,-7)*+{N}="6";
{\ar^{\Omk\ot f} "0";"2"};
{\ar_{\ell_M\ppr} "0";"4"};
{\ar^{\ell_N\ppr} "2";"6"};
{\ar_{f} "4";"6"};
{\ar@{}|\circlearrowright "0";"6"};
\endxy
\]
\end{itemize}
This follows from the naturality of $\ell\ppr$.
\end{proof}

\section{Monoidal equivalence $\M_d\simeq\BFk$}

In this section, we show {\rm (III)}.
The following has been shown in \cite{N_BisetMackey}.
\begin{fact}\label{FactEquivN}(\cite[section 6, Theorem 6.3.11]{N_BisetMackey})
An equivalence of categories $\Phi\co\M_d\to\BFk$ is given in the following way.
\begin{enumerate}
\item Let $M$ be an object in $\M_d$. Then an object $\Phi(M)$ in $\BFk$ is associated to $M$ as follows.
\begin{itemize}
\item[-] For any finite group $G$, put $\Phi(M)(G)=M(\frac{\pt}{G})$.
\item[-] For any $H$-$G$-biset $U$, put $\Phi(M)(U)=M(s_{(U)})$. Here, $s_{(U)}\in\Bcal(G,H)$ is defined as
\[ s_{(U)}=[\pth\ov{\frac{\ttt}{\prh}}{\lla}\frac{U}{H\ti G}\ov{\frac{\ttt}{\prg}}{\lra}\ptg] \]
where $U$ is regarded as an $H\ti G$-set by $(h,g)u=hug\iv$ $(\fa(h,g)\in H\ti G)$.

By the linearity, this is extended to any morphism in $\Bcal$.
\end{itemize}
\item Let $\varphi\co M\to N$ be a morphism in $\M_d$. Then a morphism $\Phi(\varphi)\co \Phi(M)\to \Phi(N)$ is defined by
\[ \Phi(\varphi)=\{ \varphi_{\ptg}\co M(\ptg)\to N(\ptg) \}_{G\in\Ob(\Bcal)}. \]
\end{enumerate}
\end{fact}

\begin{rem}
$\Phi$ sends $\Omk\in\Ob(\M_d)$ to the Burnside biset functor, i.e., the unit for the tensor product in $\BFk$.
\end{rem}

\begin{prop}\label{PropPermBF}
For any $M\in\Ob(\M_d)$ and $G\in\Ob(\Bcal)$, we have an isomorphism
\[ \varpi^{(M,G)}\co\Phi(M_{\ptg})\ov{\cong}{\lra}\Phi(M)_G \]
natural in $M$, where the right hand side is given by the Yoneda-Dress construction for biset functors defined in \cite[Definition 8.2.3]{Bouc_biset}.

Moreover, this makes the following diagram commutative for any $G\ppr$-$G$-biset $V$.
\begin{equation}\label{XYZ}
\xy
(-14,7)*+{\Phi(M_{\ptg})}="0";
(14,7)*+{\Phi(M)_G}="2";
(-14,-7)*+{\Phi(M_{\frac{\pt}{G\ppr}})}="4";
(14,-7)*+{\Phi(M)_{G\ppr}}="6";
{\ar^{\varpi^{(M,G)}}_{\cong} "0";"2"};
{\ar_{\Phi(M_{s(V)})} "0";"4"};
{\ar^{\Phi(M)_V} "2";"6"};
{\ar_{\varpi^{(M,G\ppr)}} "4";"6"};
{\ar@{}|\circlearrowright "0";"6"};
\endxy
\end{equation}
Here, $\Phi(M)_V$ denotes the morphism of biset functors given by
\[ \{ \Phi(M)(H\ti V)\co \Phi(M)(H\ti G)\to\Phi(M)(H\ti G\ppr) \}_{H\in\Ob(\Bcal)}. \]
\end{prop}
\begin{proof}
By definition, $\Phi(M)_G$ and $\Phi(M_{\ptg})$ are given as follows.
\begin{itemize}
\item[{\rm (i)}] For any $H\in\Ob(\Bcal)$, we have
\[ \Phi(M)_G(H)=M(\frac{\pt}{H\ti G}),\quad \Phi(M_{\ptg})(H)=M(\pth\ti\ptg). \]
\item[{\rm (ii)}] For any $K$-$H$-biset $U$, we have
\[ \Phi(M)_G(U)=M(s_{(U\ti G)}),\quad \Phi(M_{\ptg})(U)=M(s_{(U)}\ti\ptg), \]
where $U\ti G$ is the $(K\ti G)$-$(H\ti G)$-biset with an action given by
\begin{eqnarray*}
&(k,g_1)(u,g)(h,g_2)=(kuh,g_1gg_2)&\\
&(\fa (k,g_1)\in K\ti G,\ \fa (u,g)\in U\ti G,\ \fa (h,g_2)\in H\ti G).&
\end{eqnarray*}
\end{itemize}

Denote the isomorphism in $\C$
\[ \pth\ti\ptg\ov{\cong}{\lra}\frac{\pt}{H\ti G} \]
by $\und{\gamma}_{H,G}$. This is natural in $G$ and $H$. It induces
\[ \varpi^{(M,G)}_H=M(\Tbf_{\und{\gamma}_{H,G}})\co\Phi(M_{\ptg})(H)\ov{\cong}{\lra}\Phi(M)_G(H). \]
(For the symbol $\Tbf$, see Remark \ref{RemTRemT}.)

There exists an isomorphism in $\C$ (i.e. the composition of the following)
\[ \frac{U}{K\ti H}\ti\ptg\ov{\cong}{\lra}%
\frac{U}{K\ti H}\ti\frac{G}{G\ti G}\ov{\cong}{\lra}%
\frac{U\ti G}{(K\ti G)\ti(H\ti G)} \]
which fits into the following commutative diagram in $\C$.
\[
\xy
(-28,8)*+{\ptk\ti\ptg}="0";
(0,8)*+{\frac{U}{K\ti H}\ti\ptg}="2";
(28,8)*+{\pth\ti\ptg}="4";
(-28,-8)*+{\frac{\pt}{K\ti G}}="10";
(0,-8)*+{\frac{U\ti G}{(K\ti G)\ti(H\ti G)}}="12";
(28,-8)*+{\frac{\pt}{H\ti G}}="14";
{\ar_{\frac{\ttt}{\pro^{(K)}}\ti\ptg} "2";"0"};
{\ar^(0.56){\frac{\ttt}{\pro^{(H)}}\ti\ptg} "2";"4"};
{\ar_{\und{\gamma}_{K,G}}^{\cong} "0";"10"};
{\ar^{\cong} "2";"12"};
{\ar^{\und{\gamma}_{H,G}}_{\cong} "4";"14"};
{\ar^(0.6){\frac{\ttt}{\pro^{(K\ti G)}}} "12";"10"};
{\ar_(0.6){\frac{\ttt}{\pro^{(H\ti G)}}} "12";"14"};
{\ar@{}|\circlearrowright "0";"12"};
{\ar@{}|\circlearrowright "2";"14"};
\endxy
\]
This makes the following diagram commutative,
\[
\xy
(-18,8)*+{\Phi(M_{\ptg})(K)}="0";
(18,8)*+{\Phi(M_{\ptg})(H)}="2";
(-18,-8)*+{\Phi(M)_G(K)}="4";
(18,-8)*+{\Phi(M)_G(H)}="6";
{\ar_{\Phi(M_{\ptg})(U)} "2";"0"};
{\ar_{\varpi^{(M,G)}_K}^{\cong} "0";"4"};
{\ar^{\varpi^{(M,G)}_H}_{\cong} "2";"6"};
{\ar^{\Phi(M)_G(U)} "6";"4"};
{\ar@{}|\circlearrowright "0";"6"};
\endxy
\]
and thus gives isomorphism $\varpi^{(M,G)}=\{\varpi^{(M,G)}_H\}_{H\in\Ob(\Bcal)}\co\Phi(M_{\ptg})\ov{\cong}{\lra}\Phi(M)_G$. Naturality in $M$ can be checked in a straightforward way.

Similarly, the commutativity of $(\ref{XYZ})$ follows from the existence of the following commutative diagram in $\C$.
\[
\xy
(-28,8)*+{\ptk\ti\ptg}="0";
(0,8)*+{\ptk\ti \frac{V}{G\ti G\ppr}}="2";
(28,8)*+{\ptk\ti\frac{\pt}{G\ppr}}="4";
(-28,-8)*+{\frac{\pt}{K\ti G}}="10";
(0,-8)*+{\frac{K\ti V}{(K\ti G)\ti(K\ti G\ppr)}}="12";
(28,-8)*+{\frac{\pt}{K\ti G\ppr}}="14";
{\ar_{\ptk\ti\frac{\ttt}{\pro^{(G)}}} "2";"0"};
{\ar^(0.56){\ptk\ti\frac{\ttt}{\pro^{(G\ppr)}}} "2";"4"};
{\ar_{\und{\gamma}_{K,G}}^{\cong} "0";"10"};
{\ar^{\cong} "2";"12"};
{\ar^{\und{\gamma}_{K,G\ppr}}_{\cong} "4";"14"};
{\ar^(0.6){\frac{\ttt}{\pro^{(K\ti G)}}} "12";"10"};
{\ar_(0.6){\frac{\ttt}{\pro^{(K\ti G\ppr)}}} "12";"14"};
{\ar@{}|\circlearrowright "0";"12"};
{\ar@{}|\circlearrowright "2";"14"};
\endxy
\]
\end{proof}

\begin{prop}\label{PropMonMB}
For any $M,N\in\Ob(\M_d)$, there is a natural isomorphism of biset functors
\[ \xi^{(M,N)}\co\Phi(\HH_d(M,N)))\ov{\cong}{\lra}\Hcal(\Phi(M),\Phi(N)). \]
Here, $\Hcal\co(\BFk)^{\op}\ti\BFk\to\BFk$ denotes the internal Hom functor for biset functors defined in \cite[Definition 8.3.1]{Bouc_biset}.
\end{prop}
\begin{proof}
Put $P=\Phi(\HH_d(M,N))$ and $Q=\Hcal(\Phi(M),\Phi(N))$ for simplicity.
For any $G\in\Ob(\Bcal)$, we define $\xi^{(M,N)}_G$ to be the composition of
\begin{eqnarray*}
P(G)&=&\HH_d(M,N)(\ptg)%
\ =\ \M_d(M,N_{\ptg})\\
&\underset{\Phi}{\ov{\cong}{\lra}}&\BFk(\Phi(M),\Phi(N_{\ptg}))%
\underset{\varpi^{(N,G)}\ci-}{\ov{\cong}{\lra}}\BFk(\Phi(M),\Phi(N)_G)=Q(G),
\end{eqnarray*}
where $\Phi\co\M_d(M,N_{\ptg})\ov{\cong}{\lra}\BFk(\Phi(M),\Phi(N_{\ptg}))$ is the isomorphism between the sets of morphisms induced from the equivalence $\Phi\co\M_d\ov{\simeq}{\lra}\BFk$.

It remains to show the compatibility with respect to morphisms in $\Bcal$.
Let $U$ be any $H$-$G$-biset. The homomorphisms
\[ P(U)\co P(G)\to P(H)\quad\text{and}\quad Q(U)\co Q(G)\to Q(H)\]
are given by
\begin{eqnarray*}
\HH_d(M,N)(s_{(U)})=(N_{s_{(U)}}\ci-)&\co&\M_d(M,N_{\ptg})\to\M_d(M,N_{\pth}),\\
\Hcal(\Phi(M),\Phi(N))(U)=(\Phi(N)_U\ci-)&\co&\BFk(\Phi(M),\Phi(N)_G)\to\BFk(\Phi(M),\Phi(N)_H)
\end{eqnarray*}
by their definitions.
Thus the diagram $(\ref{XYZ})$ induces the following commutative diagram.
\[
\xy
(-42,8)*+{\M_d(M,N_{\ptg})}="0";
(-4,8)*+{\BFk(\Phi(M),\Phi(N_{\ptg}))}="2";
(42,8)*+{\BFk(\Phi(M),\Phi(N)_G)}="4";
(-42,-8)*+{\M_d(M,N_{\pth})}="10";
(-4,-8)*+{\BFk(\Phi(M),\Phi(N_{\pth}))}="12";
(42,-8)*+{\BFk(\Phi(M),\Phi(N)_H)}="14";
{\ar^(0.44){\Phi}_(0.44){\cong} "0";"2"};
{\ar^{\varpi^{(N,G)}\ci-}_{\cong} "2";"4"};
{\ar_{P(U)}^{=(N_{s_{(U)}}\ci-)} "0";"10"};
{\ar^{\Phi(N_{s_{(U)}})\ci-} "2";"12"};
{\ar_{Q(U)}^{=(\Phi(N)_U\ci-)} "4";"14"};
{\ar_(0.44){\Phi}^(0.44){\cong} "10";"12"};
{\ar_{\varpi^{(N,H)}\ci-}^{\cong} "12";"14"};
{\ar@{}|\circlearrowright "0";"12"};
{\ar@{}|\circlearrowright "2";"14"};
\endxy
\]
This means $Q(U)\ci\xi^{(M,N)}_G=\xi^{(M,N)}_H\ci P(U)$. By linearity, it follows that $\xi^{(M,N)}$ is compatible with any morphism in $\Bcal$.
\end{proof}

\begin{thm}\label{ThmMonMB}
The equivalence $\Phi\co\M_d\ov{\simeq}{\lra}\BFk$ is a monoidal equivalence.
\end{thm}
\begin{proof}
Since $\Phi(\Omk)$ is equal to the Burnside biset functor, $\Phi$ preserves the units for the tensor products.
By the adjoint property, it remains to confirm that $\Phi$ satisfies the following compatibility conditions.
\begin{enumerate}
\item For any $M\in\Ob(\M_d)$, for the natural morphisms
\[ \Omk\ov{v}{\lra}\HH_d(M,M),\quad\Phi(\Omk)\ov{v\ppr}{\lra}\Hcal(\Phi(M),\Phi(M)), \]
the following diagram is commutative.
\[
\xy
(-16,6)*+{\Phi(\Omk)}="0";
(16,6)*+{\Phi(\HH_d(M,M))}="2";
(-2,-6)*+{}="3";
(16,-8)*+{\Hcal(\Phi(M),\Phi(M))}="4";
{\ar^(0.4){\Phi(v)} "0";"2"};
{\ar_(0.4){v\ppr} "0";"4"};
{\ar^(0.4){\xi^{(M,M)}}_{\cong} "2";"4"};
{\ar@{}|\circlearrowright "2";"3"};
\endxy
\]
\item For any $M\in\Ob(\M_d)$, for the natural morphisms
\[ \HH_d(\Omk,M)\ov{j}{\lra}M,\quad\Hcal(\Phi(\Omk),\Phi(M))\ov{j\ppr}{\lra}\Phi(M), \]
the following diagram is commutative.
\[
\xy
(-18,6)*+{\Phi(\HH_d(\Omk,M))}="0";
(18,6)*+{\Phi(M)}="2";
(2,-6)*+{}="3";
(-18,-8)*+{\Hcal(\Phi(\Omk),\Phi(M))}="4";
{\ar^{\Phi(j)} "0";"2"};
{\ar_{\xi^{(\Omk,M)}}^{\cong} "0";"4"};
{\ar_{j\ppr} "4";"2"};
{\ar@{}|\circlearrowright "0";"3"};
\endxy
\]
\item For any $L,M,N\in\Ob(\M_d)$, the following diagram is commutative.
\[
\xy
(-40,8)*+{\Phi(\HH_d(M,N))}="0";
(0,22)*+{\Phi(\HH_d(\HH_d(L,M),\HH_d(L,N)))}="2";
(40,8)*+{\Hcal(\Phi(\HH_d(L,M)),\Phi(\HH_d(L,N)))}="4";
(-40,-8)*+{\Hcal(\Phi(M),\Phi(N))}="10";
(0,-22)*+{\Hcal(\Hcal(\Phi(L),\Phi(M)),\Hcal(\Phi(L),\Phi(M))))}="12";
(40,-8)*+{\Hcal(\Phi(\HH_d(L,M)),\Hcal(\Phi(L),\Phi(N)))}="14";
{\ar^{} "0";"2"};
{\ar^{} "2";"4"};
{\ar_{\xi^{(M,N)}} "0";"10"};
{\ar^{\Hcal(\Phi(\HH_d(L,M)),\xi^{(L,N)})} "4";"14"};
{\ar^{} "10";"12"};
{\ar_{\qquad\Hcal(\xi^{(L,M)},\Hcal(\Phi(L),\Phi(M)))} "12";"14"};
{\ar@{}|\circlearrowright "0";"14"};
\endxy
\]
\end{enumerate}

\bigskip

{\rm (1)} It suffices to show the commutativity of
\begin{equation}\label{LastA}
\xy
(-22,6)*+{\Omk(\ptg)}="0";
(22,6)*+{\M_d(M,M_{\ptg})}="2";
(-22,-6)*+{\BFk(\Phi(M),\Phi(M)_G)}="4";
(22,-6)*+{\BFk(\Phi(M),\Phi(M_{\ptg}))}="6";
{\ar^{v_{\ptg}} "0";"2"};
{\ar_{v\ppr_G} "0";"4"};
{\ar^{\Phi} "2";"6"};
{\ar^{\varpi^{(M,G)}\ci-} "6";"4"};
{\ar@{}|\circlearrowright "0";"6"};
\endxy
\end{equation}
for any $G\in\Ob(\Bcal)$. For any $A\in\Ob(\Gs)$, those $v_{\ptg}(A)$ and $v_G\ppr(A)$ are given by
\[ v_{\ptg}(A)=M_{s_{(A)}},\quad v\ppr_G(A)=\Phi(M)_A \]
where $A$ is viewed as an $G$-$e$-biset.
Thus the commutativity of $(\ref{LastA})$ follows from the commutativity of $(\ref{XYZ})$.

{\rm (2)}  It suffices to show the commutativity of
\begin{equation}\label{LastB}
\xy
(-24,6)*+{\M_d(\Omk,M_{\ptg})}="0";
(24,6)*+{M(\ptg)}="2";
(-24,-6)*+{\BFk(\Phi(\Omk),\Phi(M_{\ptg}))}="4";
(24,-6)*+{\BFk(\Phi(\Omk),\Phi(M)_G)}="6";
{\ar^{j_{\ptg}} "0";"2"};
{\ar_{\Phi} "0";"4"};
{\ar_{j\ppr_G} "6";"2"};
{\ar_{\varpi^{(M,G)}\ci-} "4";"6"};
{\ar@{}|\circlearrowright "0";"6"};
\endxy
\end{equation}
for any $G\in\Ob(\Bcal)$. For any $\varphi\in\M_d(\Omk,M_{\ptg})$ and any $\psi\in\BFk(\Phi(\Omk),\Phi(M)_G)$, those $j_{\ptg}(\varphi),j\ppr_G(\psi)\in M(\ptg)$ are given by
\[ j_{\ptg}(\varphi)=\varphi_{\pte}({\pt}),\quad j\ppr_G(\psi)=\psi_e(\pt), \]
where $\pt\in\Omk(e)$ denotes the trivial one-point set as before.
Thus the commutativity of $(\ref{LastB})$ follows from the definition of $\Phi(\varphi)$.

{\rm (3)} is also shown by the evaluation.

\end{proof}

\begin{cor}\label{Cor2}
The category of Green biset functors is equivalent to $\GreendCk$.
\end{cor}
\begin{proof}
By Proposition \ref{PropDeflMon}, a deflative Green functor is nothing but a monoid object in the monoidal category $\M_d$. Thus this immediately follows from Theorem \ref{ThmMonMB}, since a monoidal equivalence induces an equivalence between the categories of monoids.
\end{proof}

\section*{Acknowledgement}
This article has been written when the author was staying at LAMFA, l'Universit\'{e} de Picardie-Jules Verne, by the support of JSPS Postdoctoral Fellowships for Research Abroad. He wishes to thank the hospitality of Professor Serge Bouc, Professor Radu Stancu and the members of LAMFA.


\begin{thebibliography}{10}                                                      
\bibitem{AM} Aguiar, M.; Mahajan, S.: \emph{Monoidal functors, species and Hopf algebras}. With forewords by Kenneth Brown and Stephen Chase and Andr\'{e} Joyal. CRM Monograph Series, \textbf{29}. American Mathematical Society, Providence, RI, 2010.

\bibitem{Barwick} Barwick, C.: \emph{Spectral Mackey functors and equivariant algebraic $K$-theory {\rm (I)}}, arXiv:1404.0108.

%\bibitem{B-B}Bley, W.; Boltje, R.: \emph{Cohomological Mackey functors in number theory}. J. Number Theory \textbf{105} (2004) 1--37.

%\bibitem{Boltje}Boltje, R.: \emph{Mackey functors and related structures in representation theory and number theory}, Habilitation-Thesis, Universit\"{a}t Augsburg (1995).

\bibitem{Borceux} Borceux, F.: \emph{Handbook of categorical algebra. 1. Basic category theory}, Encyclopedia of Mathematics and its Applications, \textbf{50}. Cambridge University Press, Cambridge, 1994. xvi+345 pp.

\bibitem{Bouc_fused} Bouc, S.: \emph{Fused Mackey functors}, Geom. Dedicata \textbf{176} (2015) 225--240. .

\bibitem{Bouc_biset} Bouc, S.: \emph{Biset functors for finite groups}, Lecture Notes in Mathematics, 1990, Springer-Verlag, Berlin (2010).

\bibitem{Bouc} Bouc, S.: \emph{Green functors and $G$-sets}, Lecture Notes in Mathematics, 1671, Springer-Verlag, Berlin (1997).

\bibitem{Cisinski}Cisinski, D-C.: \emph{Images directes cohomologiques dans les cat\'{e}gories de mod\`{e}les}. (French) [Cohomological direct images in model categories] Ann. Math. Blaise Pascal \textbf{10} (2003) no. 2, 195--244. 

\bibitem{Coskun} Co\c{s}kun, O.: \emph{Inducing native Mackey functors to biset functors}. J. Pure Appl. Algebra \textbf{219} (2015) no. 6, 2359--2380.

%\bibitem {Day} Day, B.: \emph{On closed categories of functors}. 1970 Reports of the Midwest Category Seminar, IV pp. 1--38 Lecture Notes in Mathematics, Vol. 137 Springer, Berlin.

%\bibitem {Dupont} Dupont, M.: \emph{Abelian categories in dimension 2}, arXiv:0809.1760.

%\bibitem{Dupont_Vitale} Dupont, M.; Vitale, E.M.: \emph{Proper factorization systems in 2-categories}, J. Pure Appl. Algebra \textbf{179} (2003) no. 1--2, 65--86.

\bibitem{Groth} Groth, M.: \emph{Derivators, pointed derivators and stable derivators}. Algebr. Geom. Topol. \textbf{13} (2013) no. 1, 313--374.

\bibitem{HTW} Hambleton, I.; Taylor, L.R.; Williams, E.B.: \emph{Mackey functors and bisets}. Geom. Dedicata \textbf{148} (2010) 157--174.

\bibitem{Ibarra} Ibarra, J.: \emph{A generalization of the category of biset functors}, preprint.

\bibitem{JS} Joyal, Street.:  \emph{Pullbacks equivalent to pseudopullbacks}, Cahiers Topologie G\'{e}om. Diff\'{e}rentielle Cat\'{e}g. \textbf{34} (1993) no. 2, 153--156.

\bibitem{Leinster} Leinster, T.: \emph{Basic bicategories}, arXiv:math/9810017.

%\bibitem{Lindner} Lindner, H.: \emph{A remark on Mackey-functors}, Manuscripta math. \textbf{18} (1976), 273-278.

\bibitem{MacLane} Mac Lane, S.: \emph{Categories for the working mathematician}. Second edition. Graduate Texts in Mathematics, \textbf{5}. Springer-Verlag, New York, (1998). xii+314 pp.

\bibitem{MM} Mac Lane, S.; Moerdijk, I.: \emph{Sheaves in geometry and logic. A first introduction to topos theory}. Universitext. Springer-Verlag, New York, 1994. xii+629 pp.

\bibitem{N_BisetMackey} Nakaoka, H.: \emph{A Mackey-functor theoretic interpretation of biset functors},  Adv. Math. \textbf{289} (2016) 603--684.

\bibitem{N_DerTam} Nakaoka, H.: \emph{Partial Tambara structure on the Burnside biset functor, induced from a derivator-like system of adjoint triplets},  J. Algebra \textbf{451} (2016) 166--207.


%\bibitem{N_Coh} Nakaoka, H.: \emph{Cohomology theory in 2-categories}, Theory and Applications of Categories \textbf{20} (2008) 542--604.

\bibitem{PS} Panchadcharam, E.; Street, R.: \emph{Mackey functors on compact closed categories}. J. Homotopy Relat. Struct. \textbf{2} (2007) no. 2, 261--293.

\bibitem{Romero} Romero, N.: \emph{Simple modules over Green biset functors}. J. Algebra \textbf{367} (2012) 203--221.

%\bibitem{Yoshida2} Yoshida, T.: \emph{On the unit groups of Burnside rings}. J. Math. Soc. Japan \textbf{42} (1990) no. 1, 31--64.
\end{thebibliography}
\end{document}